\documentclass[12pt]{article}

\usepackage{xfrac}
\usepackage{titling}
\usepackage{appendix}
\usepackage[numbers,square,comma]{natbib}
\usepackage{float}
\usepackage[pdfpagelabels]{hyperref}
\usepackage[nointegrals]{wasysym}
\usepackage{amsmath}
\usepackage{amssymb}
\usepackage{amsthm}
\usepackage{graphicx}
\usepackage{geometry}
\usepackage{tikz}
\usepackage{tkz-graph}
\usepackage{mathptmx}
\usepackage{chngcntr}
\usepackage{thmtools}
\usepackage{thm-restate}
\usepackage{authblk}
\usetikzlibrary{shapes}

\hypersetup{%
    pdftitle={Firefighting with a Distance-Based Restriction},
    pdfauthor={Andrea C. Burgess, John Marcoux, David A. Pike},                                                   
    pdfkeywords={graph theory, combinatorics},             
    colorlinks = true,                                                          
    linkcolor = black,                                                          
    anchorcolor = black,                                                        
    citecolor = black,                                                          
    filecolor = black,                                                          
    urlcolor = black,                                                           
    pdfprintscaling=None                                                        
 }

\counterwithin{figure}{section}

\title{Firefighting with a Distance-Based Restriction}

\author[1]{Andrea C. Burgess}
\author[2]{John Marcoux}
\author[3]{David A. Pike}
\affil[1]{Department of Mathematics and Statistics, University of New Brunswick, Saint John, NB, E2L 4L5, Canada. \texttt{andrea.burgess@unb.ca}}
\affil[2]{Department of Mathematics, Toronto Metropolitan University, Toronto, ON, M5B 2K3, Canada. \texttt{jmarcoux@torontomu.ca}}
\affil[3]{Department of Mathematics and Statistics, Memorial University of Newfoundland, St. John's, NL, A1C 5S7, Canada. \texttt{dapike@mun.ca}}
\date{\today}

\newtheorem{theorem}{Theorem}[section]
\newtheorem{corollary}[theorem]{Corollary}
\newtheorem{lemma}[theorem]{Lemma}

\theoremstyle{definition}

\begin{document}

\maketitle
\medskip

\begin{abstract}
    In the classic version of the game of firefighter, on the first turn a fire breaks out on a vertex in a graph $G$ and then $k$ firefighters protect $k$ vertices. On each subsequent turn, the fire spreads to the collective unburnt neighbourhood of all the burning vertices and the firefighters again protect $k$ vertices. Once a vertex has been burnt or protected it remains that way for the rest of the game. A common objective with respect to some infinite graph $G$ is to determine how many firefighters are necessary to stop the fire from spreading after a finite number of turns, commonly referred to as \textit{containing} the fire. We introduce the concept of \textit{distance-restricted firefighting} where the firefighters' movement is restricted so they can only move up to some fixed distance $d$ per turn rather than being able to move without restriction. We establish some general properties of this new game in contrast to properties of the original game, and we investigate specific cases of the distance-restricted game on the infinite square, strong, and hexagonal grids. We conjecture that two firefighters are insufficient on the infinite square grid when $d=2$, and we pose some questions about how many firefighters are required in general when $d=1$.
\end{abstract}

\section{Introduction} \label{sec:intro}

In this paper we introduce the concept of the distance-restricted firefighter game based on Hartnell's firefighter game introduced in~\cite{hartnell1995firefighter}. The original game is a turn-based game played on a graph where fire is a proxy for some contagion that spreads from a set of `burning' vertices to their collective neighbourhoods on each turn. For our purposes we consider the game to be a two player game where the players are an arsonist and a brigade of firefighters\footnote{Note that this is in contrast to how the problem is phrased in the case of decision problems concerning finite graphs where the convention is that the initial burning vertices are part of the problem statement.}. On turn zero, one or more vertices are selected by the arsonist as the initial burning vertices and some fixed number of firefighters are placed as well. Each firefighter protects a single vertex from the fire, the consequence being that the fire does not spread to these vertices for the remainder of the game. On each subsequent turn, the fire spreads to the unprotected neighbourhoods of the burning vertices and the firefighters protect a new set of vertices. The firefighters are generally trying to achieve a predetermined goal, which in our case will be to reach a point where the fire can no longer spread. This goal is referred to as \textit{containing} the fire and is most relevant to infinite graphs, which will be our focus. We will also need the definition of \textit{average} firefighting from~\cite{messinger2005firefighting, messinger2008average}. Here the number of firefighters is not constant. Instead, there is a fixed sequence of positive integers $[a_0,\ldots,a_{T-1}]$ and the number of firefighters available on turn $t$ is $a_{(t \text{\:mod\:} T)}$. The number of firefighters used is then computed as the average of the $a_i$. For a survey of results on firefighting, we recommend the paper by Finbow and MacGillivray~\cite{finbow2009firefighter} and the recent thesis by Wagner~\cite{wagner2021new}.

In the original game there is no relationship between the firefighters' positions at time $t$ and time $t+1$. Our variation of the rules, which we have dubbed \textit{distance-restricted} or \textit{distance} $d$ firefighting, limits the fire brigade by forcing each firefighter to only move a maximum of some specified distance $d$ between time $t$ and time $t+1$. We now formalize this condition as a bipartite matching problem. Let $d$ be the distance parameter, $k$ be the number of firefighters, and $k_{u,t}$ be the number of firefighters which occupy the vertex $u$ at time $t$. Define \textit{POS}$_t$ as the set of all pairs $(i,u)$ where $u$ is a vertex occupied by at least one firefighter at time $t$ and $i$ is any positive integer between $1$ and $k_{u,t}$ inclusive. Note that the cardinality of \textit{POS}$_t$ will always be $k$. For any nonnegative integer $t$ define the bipartite graph $B_t$ with partite sets \textit{POS}$_t$ and \textit{POS}$_{t+1}$ and with edges between vertices $(i,u) \in \textit{POS}_{t}$ and $(j,v) \in \textit{POS}_{t+1}$ if $d(u,v) \leq d$. Note that this distance in the definition of the edge set can be taken in the original graph or in the subgraph induced by the unburnt vertices. In general we will use the distance in the subgraph induced by the unburnt vertices, but we will occasionally contrast this with the case where the distance is computed in the entire graph. Now, suppose $B_t$ has a perfect matching $M$. Let $(1,u),(2,u),\ldots,(k_{u,t},u)$ be a set of vertices in $B_t$ from the partite set \textit{POS}$_t$, and let $(1,v_1),(2,v_2),\ldots,(k_{u,t},v_{k_{u,t}})$ be the vertices they are matched to by $M$. Then the firefighters which occupy $u$ at time $t$ can move to $v_1,v_2,\ldots,v_{k_{u,t}}$. We repeat this process with a new choice of $u$ until we have exhausted all vertices from \textit{POS}$_t$. At this point all firefighters have been moved to the desired positions taken from \textit{POS}$_{t+1}$. This process can be reversed by iteratively building a matching in $B_t$ using the firefighters' moves. Thus the firefighters have a valid move from time $t$ to time $t+1$ if and only if $B_t$ has a perfect matching. We point out here that it may seem redundant to include duplicate vertices in \textit{POS}$_t$ as vertices are never redefended or defended by more than one firefighter at once in the original game, however there are situations in the distance-restricted game where this is necessary as it can facilitate movement to another vertex.

We note that our definition of the distance-restricted game requires some refinement to make it compatible with instances where the number of firefighters available varies from turn to turn. If the number of firefighters increases by $\ell$ on any given turn, then before the firefighters move, the $\ell$ additional firefighters are placed on vertices that currently have at least one firefighter on them, and then the firefighters move as normal. If the number of firefighters decreases by $\ell$ on any given turn, then before the firefighters move, $\ell$ of them are chosen and removed from play, and then the remaining firefighters move as normal.

We now introduce some notation for defining how many firefighters are required to contain a fire on a graph. In the context of the original game, we define $f(G,u)$ to be the minimum number of firefighters needed to contain a fire that breaks out at the vertex $u$ in the graph $G$ and we refer to this quantity as the \textit{firefighter number} of $G$ for a fire at $u$. If the graph is vertex transitive, like the grids we will be focusing on, we can simply write $f(G)$ since the choice of $u$ does not matter and we refer to this simply as the firefighter number of $G$. We can then analogously define $f_d(G,u)$ by considering the distance $d$ firefighting game and refer to $f_d(G,u)$ and $f_d(G)$ as the distance-restricted firefighter number. The notation $f^*_d(G,u)$ is defined analogously to $f_d(G,u)$, except that the firefighters are allowed to move through the fire.

In the following sections, we first discuss some general properties of distance-restricted firefighting in contrast to the original game. We then continue by determining the value of $f_d(G)$ for various values of $d$ on the infinite square ($G_{\square}$), strong ($G_{\boxtimes}$), and hexagonal ($G_{\hexagon}$) grids. Special attention is given to the infinite square grid as the case when $d=2$ is of particular interest. We conjecture that in this case, two firefighters are insufficient to contain the fire. In Section~\ref{sec:hex} we outline a special case of a conjecture by Messinger~\cite{messinger2005firefighting} pertaining to the value of $f(G_{\hexagon})$ and give a brief overview of how one might attempt to prove this special case and why it is likely easier to prove than Messinger's conjecture. In the final section we pose some questions about some general properties of the game as well as some questions about the game when $d=1$ specifically.

    Our main contribution is giving a sufficiently deep exploration of attempting to determine $f_2(G_{\square})$ in order to motivate it as a potential alternative to Messinger's conjecture. By the end of this paper the status of determining $f_2(G_{\square})$ will be similarly well developed to the case of determining $f(G_{\hexagon})$. Additionally, as we mentioned above, we give a weaker version of Messinger's conjecture which if proven in the affirmative would represent one of only a few pieces of progress towards proving Messinger's conjecture in the affirmative. As these containment problems are notoriously difficult, we believe that showing any such progress towards conjectures of this type is worthwhile. Finally, while we do not explore them here, this work motivates interesting new results for this variant on finite graphs with respect to the complexity of the related decision problem and the behaviour of the expected damage function~\cite{burgess2023distancerestricted}.

\section{General Properties of Distance-Restricted Firefighting} \label{sec:gen}

It is common in pursuit evasion type games such as firefighter to study how and when subgraphs are monotonic with respect to the game being studied. For containment the answer is straightforward. As long as the vertex $u$ where the fire starts in $G$ is included in the vertex set of $H$ then $f(H,u) \leq f(G,u)$. We do not formally prove this here, but it is relatively easy to see that any edges or vertices of $G$ which are missing in $H$ can never help the fire to spread more in $H$ than it does in $G$. However, we will see that this property does not hold in the distance-restricted game. This fact will become apparent shortly after observing some properties of certain infinite grids.

As noted in~\cite{messinger2008average} the infinite square, triangular ($G_{\triangle}$) and strong grids have equal vertex sets, but nested edge sets -- that is $V(G_{\square}) = V(G_{\triangle}) = V(G_{\boxtimes})$ and $E(G_{\square}) \subset E(G_{\triangle}) \subset E(G_{\boxtimes})$. This nesting can be further extended to the infinite Hexagonal grid by noting $V(G_{\hexagon}) = V(G_{\square})$ and $E(G_{\hexagon}) \subset E(G_{\square})$ as illustrated in Figure~\ref{fig:nestedgrids1}. Clearly the illustrated grid is contained in the infinite square grid, and it is easy to observe that the grid in Figure~\ref{fig:nestedgrids1} is isomorphic to the infinite hexagonal grid. 
 
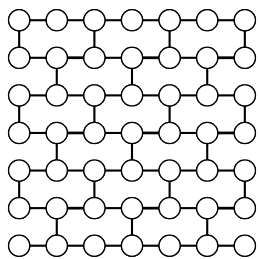
\begin{figure}[H]
    \centering
\begin{tikzpicture}[scale=0.5]
    \GraphInit[vstyle=Classic]
    \SetUpVertex[FillColor=white]

    \tikzset{VertexStyle/.append style={minimum size=8pt, inner sep=1pt}}


    \foreach \y in {0,1,...,6} {\foreach \x in {0,1,...,6} {\Vertex[x=\x,y=\y,NoLabel=true,]{V\x\y}}}
    

    \foreach[count =\i, evaluate=\i as \z using int(\i)] \y in {0,1,...,5} {\foreach \x in {0,1,...,6} {\Edge(V\x\y)(V\x\z)}}
    \foreach[count =\i, evaluate=\i as \z using int(\i)] \x in {0,1,...,5} {\foreach \y in {0,1,...,6} {\Edge(V\x\y)(V\z\y)}}
    

    \tikzset{EdgeStyle/.append style={white}}

    \foreach[count =\i, evaluate=\i as \z using int(\i)] \y in {0,1,...,5} {\foreach[count=\j, evaluate=\j as \u using int(\j)] \x in {0,1,...,5} {\Edge(V\x\y)(V\u\z)}}
    
    
    \tikzset{EdgeStyle/.append style={white}}

    \foreach[count =\i, evaluate=\i as \z using int(\i)] \x in {0,1,...,5} {\foreach[count=\j, evaluate=\j as \u using int(\j-1)] \y in {1,2,...,6} {\Edge(V\x\y)(V\z\u)}}
    
    \tikzset{EdgeStyle/.append style={white, line width = 2pt}}

    \Edge(V00)(V01)
    \Edge(V02)(V03)
    \Edge(V04)(V05)

    \Edge(V11)(V12)
    \Edge(V13)(V14)
    \Edge(V15)(V16)

    \Edge(V20)(V21)
    \Edge(V22)(V23)
    \Edge(V24)(V25)

    \Edge(V31)(V32)
    \Edge(V33)(V34)
    \Edge(V35)(V36)

    \Edge(V40)(V41)
    \Edge(V42)(V43)
    \Edge(V44)(V45)

    \Edge(V51)(V52)
    \Edge(V53)(V54)
    \Edge(V55)(V56)

    \Edge(V60)(V61)
    \Edge(V62)(V63)
    \Edge(V64)(V65)

    \tikzset{EdgeStyle/.append style={black, line width = 1pt}}

    \Edge(V11)(V21)
    \Edge(V31)(V41)
    \Edge(V51)(V61)

    \Edge(V13)(V23)
    \Edge(V33)(V43)
    \Edge(V53)(V63)

    \Edge(V15)(V25)
    \Edge(V35)(V45)
    \Edge(V55)(V65)
    
    \tikzset{VertexStyle/.append style={minimum size=8pt, inner sep=1pt}}
    \foreach \y in {0,1,...,6} {\foreach \x in {0,1,...,6} {\Vertex[x=\x,y=\y,NoLabel=true,]{V\x\y}}}

\end{tikzpicture}
\caption{$7 \times 7$ Portion of the infinite hexagonal grid as a subgraph of the infinite square grid.}
    \label{fig:nestedgrids1}
\end{figure}
    
We make note here that the vertex set of all four of these grids can be thought of as $\mathbb{Z} \times \mathbb{Z}$ with the naturally defined edge sets. This allows us to make use of the idea of positive $y$, negative $y$, positive $x$ and negative $x$ as up, down, right, and left respectively to refer to how the fire and firefighters move in the graph as well as above and below for where the fire and firefighter are with respect to one another. When we draw the grids in figures we will adopt this convention in the expected way. For further definitions from graph theory we direct the reader to Diestel's graph theory text~\cite{diestel2018graphtheory}. In particular, the definition of ray and double ray (page 210) will be used in this paper and we will denote them $P^{\mathbb{N}}$ and $P^{\mathbb{Z}}$ respectively.

    Observe in Figure~\ref{fig:nestedgrids2} that the infinite hexagonal grid contains a subdivision of itself which we will refer to as \textit{sub}($G_{\hexagon}$). The paths of length three formed by the degree two vertices and their neighbours can be replaced with single edges in order to recover the infinite hexagonal grid.

\begin{figure}[H]
    \centering
\begin{tikzpicture}[scale=0.5]
    \GraphInit[vstyle=Classic]
    \SetUpVertex[FillColor=white]

    \tikzset{VertexStyle/.append style={minimum size=8pt, inner sep=1pt}}


    \foreach \y in {0,1,...,6} {\foreach \x in {0,1,...,6} {\Vertex[x=\x,y=\y,NoLabel=true,]{V\x\y}}}
    

    \foreach[count =\i, evaluate=\i as \z using int(\i)] \y in {0,1,...,5} {\foreach \x in {0,1,...,6} {\Edge(V\x\y)(V\x\z)}}
    \foreach[count =\i, evaluate=\i as \z using int(\i)] \x in {0,1,...,5} {\foreach \y in {0,1,...,6} {\Edge(V\x\y)(V\z\y)}}
    

    \tikzset{EdgeStyle/.append style={white, line width = 2pt}}

    \foreach[count =\i, evaluate=\i as \z using int(\i)] \y in {0,1,...,5} {\foreach[count=\j, evaluate=\j as \u using int(\j)] \x in {0,1,...,5} {\Edge(V\x\y)(V\u\z)}}
    
    
    \tikzset{EdgeStyle/.append style={white}}

    \foreach[count =\i, evaluate=\i as \z using int(\i)] \x in {0,1,...,5} {\foreach[count=\j, evaluate=\j as \u using int(\j-1)] \y in {1,2,...,6} {\Edge(V\x\y)(V\z\u)}}
    
    \tikzset{EdgeStyle/.append style={white}}

    \Edge(V00)(V01)
    \Edge(V02)(V03)
    \Edge(V04)(V05)

    \Edge(V11)(V12)
    \Edge(V13)(V14)
    \Edge(V15)(V16)

    \Edge(V20)(V21)
    \Edge(V22)(V23)
    \Edge(V24)(V25)

    \Edge(V31)(V32)
    \Edge(V33)(V34)
    \Edge(V35)(V36)

    \Edge(V40)(V41)
    \Edge(V42)(V43)
    \Edge(V44)(V45)

    \Edge(V51)(V52)
    \Edge(V53)(V54)
    \Edge(V55)(V56)

    \Edge(V60)(V61)
    \Edge(V62)(V63)
    \Edge(V64)(V65)

    \tikzset{EdgeStyle/.append style={white}}

    \Edge(V11)(V21)
    \Edge(V31)(V41)
    \Edge(V51)(V61)

    \Edge(V13)(V23)
    \Edge(V33)(V43)
    \Edge(V53)(V63)

    \Edge(V15)(V25)
    \Edge(V35)(V45)
    \Edge(V55)(V65)
    
    \tikzset{VertexStyle/.append style={minimum size=8pt, inner sep=1pt}}
    \foreach \y in {0,1,...,6} {\foreach \x in {0,1,...,6} {\Vertex[x=\x,y=\y,NoLabel=true,]{V\x\y}}}

\end{tikzpicture}
\caption{$7 \times 7$ Portion of \textit{sub}($G_{\hexagon}$) as a subgraph of the infinite hexagonal grid.}
    \label{fig:nestedgrids2}
\end{figure}

Observe from Figure~\ref{fig:nestedgrids3} that \textit{sub}($G_{\hexagon}$) only requires one firefighter when $d \geq 11$, however \textit{sub}($G_{\hexagon}$) contains a double ray\footnote{For our purposes, consider any set of vertices $\{(x,y) \mid x \in \mathbb{Z}\}$ for an appropriate choice of $y$.} which requires two firefighters for any value of $d$. We note here that there are many diagrams in this article that all follow the same labelling conventions. The firefighters are labelled using square black vertices. The burnt vertices are labelled using circular orange vertices. The vertices where the fire starts are diamond shaped and the vertices where the firefighter starts are star shaped. We can now see that the distance-restricted firefighter number of a subgraph is not bounded above by the distance-restricted firefighter number of the original graph. In fact, Figure~\ref{fig:layered_wheel} and Theorem~\ref{thm:arb_diff_2} show that the ratio and difference of $f_d(G)$ and $f_d(H)$ for $H$ a subgraph of $G$ can both be unbounded. 

\begin{figure}[H]
    \centering
\begin{tikzpicture}[scale=0.5]
    \SetUpVertex[FillColor=white]

    \tikzset{VertexStyle/.append style={minimum size=8pt, inner sep=1pt}}


    \foreach \y in {0,1,...,6} {\foreach \x in {0,1,...,6} {\Vertex[x=\x,y=\y,NoLabel=true,]{V\x\y}}}
    

    \foreach[count =\i, evaluate=\i as \z using int(\i)] \y in {0,1,...,5} {\foreach \x in {0,1,...,6} {\Edge(V\x\y)(V\x\z)}}
    \foreach[count =\i, evaluate=\i as \z using int(\i)] \x in {0,1,...,5} {\foreach \y in {0,1,...,6} {\Edge(V\x\y)(V\z\y)}}
    

    \tikzset{EdgeStyle/.append style={white, line width = 2pt}}

    \foreach[count =\i, evaluate=\i as \z using int(\i)] \y in {0,1,...,5} {\foreach[count=\j, evaluate=\j as \u using int(\j)] \x in {0,1,...,5} {\Edge(V\x\y)(V\u\z)}}
    
    
    \tikzset{EdgeStyle/.append style={white}}

    \foreach[count =\i, evaluate=\i as \z using int(\i)] \x in {0,1,...,5} {\foreach[count=\j, evaluate=\j as \u using int(\j-1)] \y in {1,2,...,6} {\Edge(V\x\y)(V\z\u)}}
    
    \tikzset{EdgeStyle/.append style={white}}

    \Edge(V00)(V01)
    \Edge(V02)(V03)
    \Edge(V04)(V05)

    \Edge(V11)(V12)
    \Edge(V13)(V14)
    \Edge(V15)(V16)

    \Edge(V20)(V21)
    \Edge(V22)(V23)
    \Edge(V24)(V25)

    \Edge(V31)(V32)
    \Edge(V33)(V34)
    \Edge(V35)(V36)

    \Edge(V40)(V41)
    \Edge(V42)(V43)
    \Edge(V44)(V45)

    \Edge(V51)(V52)
    \Edge(V53)(V54)
    \Edge(V55)(V56)

    \Edge(V60)(V61)
    \Edge(V62)(V63)
    \Edge(V64)(V65)

    \tikzset{EdgeStyle/.append style={white}}

    \Edge(V11)(V21)
    \Edge(V31)(V41)
    \Edge(V51)(V61)

    \Edge(V13)(V23)
    \Edge(V33)(V43)
    \Edge(V53)(V63)

    \Edge(V15)(V25)
    \Edge(V35)(V45)
    \Edge(V55)(V65)
    
    \tikzset{VertexStyle/.append style={minimum size=8pt, inner sep=1pt}}
    \foreach \y in {0,1,...,6} {\foreach \x in {0,1,...,6} {\Vertex[x=\x,y=\y,NoLabel=true,]{V\x\y}}}
    
    \tikzset{VertexStyle/.append style={minimum size=12pt, inner sep=0.5pt, diamond, orange, text=black}}

    \Vertex[x=3,y=2,L={\scriptsize 0}]{V32}

    \tikzset{VertexStyle/.append style={minimum size=8pt, inner sep=1pt, circle, orange, text=black}}

    \Vertex[x=3,y=3,L={\scriptsize 1}]{V33}
    \Vertex[x=4,y=2,L={\scriptsize 1}]{V42}

    \tikzset{VertexStyle/.append style={orange, text=black}}

    \Vertex[x=2,y=3,L={\scriptsize 2}]{V23}
    \Vertex[x=4,y=1,L={\scriptsize 2}]{V41}
    
    \tikzset{VertexStyle/.append style={orange, text=black}}

    \Vertex[x=5,y=1,L={\scriptsize 3}]{V41}

    \tikzset{VertexStyle/.append style={minimum size=12pt, inner sep=1pt, star, black, text=white}}

    \Vertex[x=2,y=2,L={\scriptsize 0}]{V22}
    
    \tikzset{VertexStyle/.append style={minimum size=8pt, inner sep=1pt, rectangle, black, text=white}}
    
    \Vertex[x=5,y=2,L={\scriptsize 1}]{V22}
    
    \tikzset{VertexStyle/.append style={black, text=white}}

    \Vertex[x=2,y=4,L={\scriptsize 2}]{V22}
    
    \tikzset{VertexStyle/.append style={black, text=white}}

    \Vertex[x=5,y=0,L={\scriptsize 3}]{V41}

\end{tikzpicture}
    \caption{A strategy to contain the fire in four turns on \textit{sub}($G_{\hexagon}$) when $d=11$. The case where the fire begins on a degree two vertex is trivial and thus omitted.}
    \label{fig:nestedgrids3}
\end{figure}
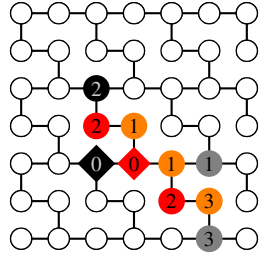

\begin{figure}[H]
        \centering 
        \begin{tikzpicture}[scale=0.5]
		    \SetUpVertex[FillColor=white]
		    
            \tikzset{VertexStyle/.append style={minimum size=8pt, inner sep=1pt}}

		    \Vertex[x=0,y=0,NoLabel=true,]{V00A}
		    
            \Vertex[x=1,y=0,NoLabel=true,]{V01}
		    \Vertex[x=0.31,y=0.95,NoLabel=true,]{V11}
		    \Vertex[x=-0.81,y=0.59,NoLabel=true,]{V21}
		    \Vertex[x=-0.81,y=-0.59,NoLabel=true,]{V31}
		    \Vertex[x=0.31,y=-0.95,NoLabel=true,]{V41}

		    \Vertex[x=2,y=0,NoLabel=true,]{V02}
		    \Vertex[x=0.62,y=1.9,NoLabel=true,]{V12}
		    \Vertex[x=-1.62,y=1.18,NoLabel=true,]{V22}
		    \Vertex[x=-1.62,y=-1.18,NoLabel=true,]{V32}
		    \Vertex[x=0.62,y=-1.9,NoLabel=true,]{V42}
		    
            \tikzset{VertexStyle/.append style={minimum size=1pt, inner sep=0pt}}

		    \Vertex[x=3,y=0,NoLabel=true,]{V03}
		    \Vertex[x=0.93,y=2.85,NoLabel=true,]{V13}
		    \Vertex[x=-2.43,y=1.77,NoLabel=true,]{V23}
		    \Vertex[x=-2.43,y=-1.77,NoLabel=true,]{V33}
		    \Vertex[x=0.93,y=-2.85,NoLabel=true,]{V43}

            \Edge(V00A)(V01)
            \Edge(V00A)(V11)
            \Edge(V00A)(V21)
            \Edge(V00A)(V31)
            \Edge(V00A)(V41)
	
            \Edge(V01)(V11)
            \Edge(V21)(V11)
            \Edge(V21)(V31)
            \Edge(V41)(V31)
            \Edge(V41)(V01)

            \Edge(V02)(V12)
            \Edge(V22)(V12)
            \Edge(V22)(V32)
            \Edge(V42)(V32)
            \Edge(V42)(V02)

            \Edge(V01)(V02)
            \Edge(V11)(V12)
            \Edge(V21)(V22)
            \Edge(V31)(V32)
            \Edge(V41)(V42)

            \tikzset{EdgeStyle/.append style={dashed}}

            \Edge(V03)(V02)
            \Edge(V13)(V12)
            \Edge(V23)(V22)
            \Edge(V33)(V32)
            \Edge(V43)(V42)

		    \tikzset{VertexStyle/.append style={minimum size=8pt, inner sep=1pt}}
            
            \tikzset{EdgeStyle/.append style={solid}}
            
            \tikzset{VertexStyle/.append style={minimum size=8pt, inner sep=1pt}}
		    
            \Vertex[x=7,y=0,NoLabel=true,]{V00}

		    \Vertex[x=8,y=0,NoLabel=true,]{V01}
		    \Vertex[x=7.31,y=0.95,NoLabel=true,]{V11}
		    \Vertex[x=6.19,y=0.59,NoLabel=true,]{V21}
		    \Vertex[x=6.19,y=-0.59,NoLabel=true,]{V31}
		    \Vertex[x=7.31,y=-0.95,NoLabel=true,]{V41}

		    \Vertex[x=9,y=0,NoLabel=true,]{V02}
		    \Vertex[x=7.62,y=1.9,NoLabel=true,]{V12}
		    \Vertex[x=5.38,y=1.18,NoLabel=true,]{V22}
		    \Vertex[x=5.38,y=-1.18,NoLabel=true,]{V32}
		    \Vertex[x=7.62,y=-1.9,NoLabel=true,]{V42}

            \tikzset{VertexStyle/.append style={minimum size=1pt, inner sep=0pt}}

		    \Vertex[x=10,y=0,NoLabel=true,]{V03}
		    \Vertex[x=7.93,y=2.85,NoLabel=true,]{V13}
		    \Vertex[x=4.57,y=1.77,NoLabel=true,]{V23}
		    \Vertex[x=4.57,y=-1.77,NoLabel=true,]{V33}
		    \Vertex[x=7.93,y=-2.85,NoLabel=true,]{V43}

            \Edge(V00)(V01)
            \Edge(V00)(V11)
            \Edge(V00)(V21)
            \Edge(V00)(V31)
            \Edge(V00)(V41)

            \Edge(V01)(V02)
            \Edge(V11)(V12)
            \Edge(V21)(V22)
            \Edge(V31)(V32)
            \Edge(V41)(V42)
		
            \tikzset{VertexStyle/.append style={minimum size=12pt, inner sep=0.5pt, orange, diamond}}
		
		    \Vertex[x=7,y=0,NoLabel=true,]{V00}
		
		    \Vertex[x=0,y=0,NoLabel=true,]{V00A}

            \tikzset{EdgeStyle/.append style={dashed}}

            \Edge(V03)(V02)
            \Edge(V13)(V12)
            \Edge(V23)(V22)
            \Edge(V33)(V32)
            \Edge(V43)(V42)

		\end{tikzpicture}
        \caption{Left: The ball of radius $3$ around the initial burning vertex of $P^{\mathbb{N}}\:\square\:C^5$ modified as described in the proof of Theorem~\ref{thm:arb_diff_2}, Right: An infinite subgraph of the graph on the left with distance-restricted firefighter number $5$.}
        \label{fig:layered_wheel}

    \end{figure}

\begin{theorem} \label{thm:arb_diff_2}
    There exist graphs $G,H$ such that $H$ is a subgraph of $G$ with $u \in V(H) \subseteq V(G)$ such that the value of $\big(f_d(H,u) - f_d(G,u)\big)$ as well as the value of $\Big(\frac{f_d(H,u)}{f_d(G,u)}\Big)$ can be arbitrarily large for any value of $d$.
\end{theorem}

\begin{proof}
    The graph in Figure~\ref{fig:layered_wheel} can be thought of as taking $P^{\mathbb{N}}\:\square\:C^m$ (the Cartesian product\footnote{$G\:\square\:H$ is the graph with vertex set $V(G) \times V(H)$ and edge set $\{(u,v)(x,y) \mid (u=x\land vy \in E(H)) \lor (ux \in E(G) \land v=y)\}$} of a ray and a cycle), adding a new vertex $u$, and making $u$ adjacent to the $m$ vertices corresponding to the end vertex of the ray. This graph has distance-restricted firefighter number $1$ for a fire breaking out at $u$ for any value of $d$ since if we start our firefighter in the $(m+1)^{th}$ cycle then the firefighter can just walk around the cycle and form a barrier. However if we consider the subgraph depicted in Figure~\ref{fig:layered_wheel}, then any fire breaking out at $u$ will require $m$ firefighters to contain the fire no matter the value of $d$. Thus both the difference and the ratio are unbounded for any value of $d$.
\end{proof}

We note that a similar theorem is proved in the finite case for the expected damage property in~\cite{burgess2023distancerestricted} (see Theorem $4.2$).

    This leads us to the question of how large the difference between the firefighter number and the distance-restricted firefighter number can be. 

\begin{theorem} \label{thm:arb_diff}
    The value of $\big(f_d(G,u) - f(G,u)\big)$ as well as the value of $\Big(\frac{f_d(G,u)}{f(G,u)}\Big)$ can be arbitrarily large for any value of $d$.
\end{theorem}

\begin{proof}
    Observe the subgraph from the proof of Theorem~\ref{thm:arb_diff_2} as depicted in Figure~\ref{fig:layered_wheel} (right). This graph clearly only requires a single firefighter in the original game regardless of where the fire starts, but requires $m$ firefighters in the distance-restricted game when the fire breaks out at the central vertex. Thus both the ratio and the difference are unbounded here.
\end{proof}

Observe that Theorem~\ref{thm:arb_diff} does not fully extend to $f^*_d$. For example, if we take any graph $G$ and a corresponding containment strategy with $f(G,u)$ firefighters where the maximum distance moved is $\delta$, then $f^*_d(G,u) = f(G,u)$ holds as long as $d \geq \delta$. The case when $d < \delta$ is less clear. We will see later that $f_3(G_{\square}) = f^*_3(G_{\square}) = f(G_{\square})$, but if we take the optimal strategy for containing the fire given in Figure $4$ of~\cite{wang2002fire}, the maximum distance moved is seven. In short, the question is only of interest when $d$ is fixed; in this case the same graphs used in Theorem~\ref{thm:arb_diff} can be used to show the ratio is unbounded. If the firefighter wishes to defend all the branches they will have to move progressively further and further to go from where they just defended on one branch to the boundary of the fire on a currently undefended branch. Thus they will certainly not be able to move to any of the undefended branches eventually if there are enough branches (i.e.~if a large enough cycle is used in the construction).

\begin{corollary}
    For any fixed value of $d$, the value of $\big(f_d(G,u) - f(G,u)\big)$ as well as the value of $\Big(\frac{f_d(G,u)}{f(G,u)}\Big)$ can be arbitrarily large. 
\end{corollary}

Observe as well that something similar can be said of Theorem~\ref{thm:arb_diff_2}. Consider $f^*_d(G,u)$ and $f^*_d(H,u)$ for $H$ a subgraph of $G$. There will exist a value of $d$ where $f^*_d(G,u) = f(G,u)$ and $f^*_d(H,u) = f(H,u)$, and so any properties of their ratio and difference are the same as in the original game. Again though, if we look at a fixed value of $d$ then the construction will work.

\begin{corollary}
    For any fixed value of $d$, there exist graphs $G,H$ such that $H$ is a subgraph of $G$ with $u \in V(H) \subseteq V(G)$ such that the value of $\big(f_d(H,u) - f_d(G,u)\big)$ as well as the value of $\Big(\frac{f_d(H,u)}{f_d(G,u)}\Big)$ can be arbitrarily large.
\end{corollary}

\section{Square Grid} \label{sec:square}

Our consideration of $G_{\square}$ can be divided into two parts: the case of $d=2$ and the case of $d\neq2$. We found the case of $d=2$ to be significantly more complicated than the case of $d\neq2$ so the majority of this section focuses on this case.

\subsection{Square Grid when $\mathbf{d\neq2}$} \label{subsec:square1}

    The case of $d=1$ on the infinite square grid was shown to require four firefighters in~\cite{chen2017continuous,days2019firefighter}. The idea in this case is that in order to stop the fire from burning along the four rays that start at the origin and go in the four cardinal directions, there must be four firefighters, and four firefighters are obviously sufficient because the fire can be trivially surrounded on turn zero.

    The case of $d=3$ also has an easy solution. Figure~\ref{fig:square_d3_2f} portrays a strategy that works with two firefighters. The two firefighters form a barrier by `spiraling clockwise'\footnote{By spiraling we loosely mean that the two firefighters always defend the next two vertices which are about to burn (with respect to the given clockwise or counterclockwise orientation). In certain cases a firefighter will not be able to reach any of these undefended vertices and in this case they will simply move as close as possible to the nearest one.} around the fire until the fire is contained.

\begin{lemma}
    Two firefighters are necessary and sufficient to contain the fire on the infinite square grid in the game with $d \geq 3$.
\end{lemma}

\begin{proof}
    Figure~\ref{fig:square_d3_2f} illustrates a strategy with two firefighters when $d=3$. The two firefighters spiral around the fire until it is contained as previously discussed. Thus two firefighters suffice when $d=3$. For any distance greater than three the same strategy can be played. Thus for any $d \geq 3$ two firefighters suffice.

    One firefighter does not suffice since if one firefighter was sufficient for $d=3$ then one firefighter would be sufficient for the original game which is not true from~\cite{wang2002fire}. 
\end{proof}

\begin{figure}[H]
    \centering
\begin{tikzpicture}[scale=0.5]
    \SetUpVertex[FillColor=white]

    \tikzset{VertexStyle/.append style={minimum size=8pt, inner sep=1pt}}

    \foreach \y in {-6,-5,...,8} {\foreach \x in {-7,-6,...,8} {\Vertex[x=\x,y=\y,NoLabel=true,]{V\x\y}}}
    \foreach[count =\i, evaluate=\i as \z using int(\i-6)] \y in {-6,-5,...,7} {\foreach \x in {-7,-6,...,8} {\Edge(V\x\y)(V\x\z)}}
    \foreach[count =\i, evaluate=\i as \z using int(\i-7)] \x in {-7,-6,...,7} {\foreach \y in {-6,-5,...,8} {\Edge(V\x\y)(V\z\y)}}
    
    \tikzset{VertexStyle/.append style={orange}}
    \tikzset{VertexStyle/.append style={diamond, minimum size = 12pt, text=black, inner sep=0.5pt}}
    
    \Vertex[x=0,y=0,L={\scriptsize 0}]{0}
    
    \tikzset{VertexStyle/.append style={orange, circle, minimum size = 8pt, text=black, inner sep=1pt}}
   
    \Vertex[x=1,y=0,L={\scriptsize 1}]{V00}
    \Vertex[x=0,y=1,L={\scriptsize 1}]{V00}

    \tikzset{VertexStyle/.append style={orange,text=black}}

    \Vertex[x=2,y=0,L={\scriptsize 2}]{V00}
    \Vertex[x=1,y=1,L={\scriptsize 2}]{V00}
    \Vertex[x=1,y=-1,L={\scriptsize 2}]{V00}
    
    \tikzset{VertexStyle/.append style={orange,text=black}}
    
    \Vertex[x=1,y=-2,L={\scriptsize 3}]{V00}
    \Vertex[x=3,y=0,L={\scriptsize 3}]{V00}
    \Vertex[x=2,y=-1,L={\scriptsize 3}]{V00}

    \tikzset{VertexStyle/.append style={orange,text=black}}

    \Vertex[x=1,y=-3,L={\scriptsize 4}]{V00}
    \Vertex[x=3,y=-1,L={\scriptsize 4}]{V00}
    \Vertex[x=2,y=-2,L={\scriptsize 4}]{V00}
    \Vertex[x=0,y=-2,L={\scriptsize 4}]{V00}

    \tikzset{VertexStyle/.append style={orange,text=black}}

    \Vertex[x=-1,y=-2,L={\scriptsize 5}]{V00}
    \Vertex[x=0,y=-3,L={\scriptsize 5}]{V00}
    \Vertex[x=1,y=-4,L={\scriptsize 5}]{V00}
    \Vertex[x=2,y=-3,L={\scriptsize 5}]{V00}

    \tikzset{VertexStyle/.append style={orange,text=black}}

    \Vertex[x=-2,y=-2,L={\scriptsize 6}]{V00}
    \Vertex[x=-1,y=-1,L={\scriptsize 6}]{V00}
    \Vertex[x=-1,y=-3,L={\scriptsize 6}]{V00}
    \Vertex[x=0,y=-4,L={\scriptsize 6}]{V00}
    \Vertex[x=1,y=-5,L={\scriptsize 6}]{V00}

    \tikzset{VertexStyle/.append style={orange,text=black}}

    \Vertex[x=-3,y=-2,L={\scriptsize 7}]{V00}
    \Vertex[x=-2,y=-1,L={\scriptsize 7}]{V00}
    \Vertex[x=-2,y=-3,L={\scriptsize 7}]{V00}
    \Vertex[x=-1,y=-4,L={\scriptsize 7}]{V00}
    \Vertex[x=0,y=-5,L={\scriptsize 7}]{V00}

    \tikzset{VertexStyle/.append style={orange,text=black}}

    \Vertex[x=-4,y=-2,L={\scriptsize 8}]{V00}
    \Vertex[x=-3,y=-1,L={\scriptsize 8}]{V00}
    \Vertex[x=-2,y=0,L={\scriptsize 8}]{V00}
    \Vertex[x=-3,y=-3,L={\scriptsize 8}]{V00}
    \Vertex[x=-2,y=-4,L={\scriptsize 8}]{V00}

    \tikzset{VertexStyle/.append style={orange,text=black}}

    \Vertex[x=-5,y=-2,L={\scriptsize 9}]{V00}
    \Vertex[x=-4,y=-1,L={\scriptsize 9}]{V00}
    \Vertex[x=-2,y=1,L={\scriptsize 9}]{V00}
    \Vertex[x=-3,y=0,L={\scriptsize 9}]{V00}
    \Vertex[x=-4,y=-3,L={\scriptsize 9}]{V00}

    \tikzset{VertexStyle/.append style={orange,text=black,inner sep=0.1pt}}

    \Vertex[x=-2,y=2,L={\scriptsize 10}]{V00}
    \Vertex[x=-3,y=1,L={\scriptsize 10}]{V00}
    \Vertex[x=-4,y=0,L={\scriptsize 10}]{V00}
    \Vertex[x=-5,y=-1,L={\scriptsize 10}]{V00}
    \Vertex[x=-6,y=-2,L={\scriptsize 10}]{V00}

    \tikzset{VertexStyle/.append style={orange,text=black}}

    \Vertex[x=-1,y=2,L={\scriptsize 11}]{V00}
    \Vertex[x=-2,y=3,L={\scriptsize 11}]{V00}
    \Vertex[x=-3,y=2,L={\scriptsize 11}]{V00}
    \Vertex[x=-4,y=1,L={\scriptsize 11}]{V00}
    \Vertex[x=-5,y=0,L={\scriptsize 11}]{V00}
    \Vertex[x=-6,y=-1,L={\scriptsize 11}]{V00}

    \tikzset{VertexStyle/.append style={orange,text=black}}

    \Vertex[x=-1,y=3,L={\scriptsize 12}]{V00}
    \Vertex[x=-2,y=4,L={\scriptsize 12}]{V00}
    \Vertex[x=-3,y=3,L={\scriptsize 12}]{V00}
    \Vertex[x=-4,y=2,L={\scriptsize 12}]{V00}
    \Vertex[x=-5,y=1,L={\scriptsize 12}]{V00}

    \tikzset{VertexStyle/.append style={orange,text=black}}

    \Vertex[x=0,y=3,L={\scriptsize 13}]{V00}
    \Vertex[x=-1,y=4,L={\scriptsize 13}]{V00}
    \Vertex[x=-2,y=5,L={\scriptsize 13}]{V00}
    \Vertex[x=-3,y=4,L={\scriptsize 13}]{V00}
    \Vertex[x=-4,y=3,L={\scriptsize 13}]{V00}

    \tikzset{VertexStyle/.append style={orange,text=black}}

    \Vertex[x=1,y=3,L={\scriptsize 14}]{V00}
    \Vertex[x=0,y=4,L={\scriptsize 14}]{V00}
    \Vertex[x=-1,y=5,L={\scriptsize 14}]{V00}
    \Vertex[x=-2,y=6,L={\scriptsize 14}]{V00}
    \Vertex[x=-3,y=5,L={\scriptsize 14}]{V00}

    \tikzset{VertexStyle/.append style={orange,text=black}}

    \Vertex[x=2,y=3,L={\scriptsize 15}]{V00}
    \Vertex[x=1,y=4,L={\scriptsize 15}]{V00}
    \Vertex[x=0,y=5,L={\scriptsize 15}]{V00}
    \Vertex[x=-1,y=6,L={\scriptsize 15}]{V00}
    \Vertex[x=-2,y=7,L={\scriptsize 15}]{V00}

    \tikzset{VertexStyle/.append style={orange,text=black}}

    \Vertex[x=2,y=2,L={\scriptsize 16}]{V00}
    \Vertex[x=3,y=3,L={\scriptsize 16}]{V00}
    \Vertex[x=2,y=4,L={\scriptsize 16}]{V00}
    \Vertex[x=1,y=5,L={\scriptsize 16}]{V00}
    \Vertex[x=0,y=6,L={\scriptsize 16}]{V00}
    \Vertex[x=-1,y=7,L={\scriptsize 16}]{V00}

    \tikzset{VertexStyle/.append style={orange,text=black}}

    \Vertex[x=3,y=2,L={\scriptsize 17}]{V00}
    \Vertex[x=4,y=3,L={\scriptsize 17}]{V00}
    \Vertex[x=3,y=4,L={\scriptsize 17}]{V00}
    \Vertex[x=2,y=5,L={\scriptsize 17}]{V00}
    \Vertex[x=1,y=6,L={\scriptsize 17}]{V00}

    \tikzset{VertexStyle/.append style={orange,text=black}}

    \Vertex[x=4,y=2,L={\scriptsize 18}]{V00}
    \Vertex[x=5,y=3,L={\scriptsize 18}]{V00}
    \Vertex[x=4,y=4,L={\scriptsize 18}]{V00}
    \Vertex[x=3,y=5,L={\scriptsize 18}]{V00}

    \tikzset{VertexStyle/.append style={orange,text=black}}

    \Vertex[x=4,y=1,L={\scriptsize 19}]{V00}
    \Vertex[x=5,y=2,L={\scriptsize 19}]{V00}
    \Vertex[x=6,y=3,L={\scriptsize 19}]{V00}
    \Vertex[x=5,y=4,L={\scriptsize 19}]{V00}

    \tikzset{VertexStyle/.append style={orange,text=black}}

    \Vertex[x=5,y=1,L={\scriptsize 20}]{V00}
    \Vertex[x=6,y=2,L={\scriptsize 20}]{V00}
    \Vertex[x=7,y=3,L={\scriptsize 20}]{V00}

    \tikzset{VertexStyle/.append style={orange,text=black}}

    \Vertex[x=5,y=0,L={\scriptsize 21}]{V00}
    \Vertex[x=6,y=1,L={\scriptsize 21}]{V00}
    \Vertex[x=7,y=2,L={\scriptsize 21}]{V00}

    \tikzset{VertexStyle/.append style={orange,text=black}}

    \Vertex[x=5,y=-1,L={\scriptsize 22}]{V00}
    \Vertex[x=6,y=0,L={\scriptsize 22}]{V00}

    \tikzset{VertexStyle/.append style={orange,text=black}}

    \Vertex[x=5,y=-2,L={\scriptsize 23}]{V00}

    \tikzset{VertexStyle/.append style={orange,text=black}}

    \Vertex[x=4,y=-2,L={\scriptsize 24}]{V00}

    \tikzset{VertexStyle/.append style={black,text=white}}
    \tikzset{VertexStyle/.append style={star, minimum size = 12pt}}

    \Vertex[x=-1,y=0,L={\scriptsize 0}]{V00}
    \Vertex[x=0,y=-1,L={\scriptsize 0}]{V00}
    
    \tikzset{VertexStyle/.append style={rectangle, minimum size = 8pt}}
    
    \tikzset{VertexStyle/.append style={black,text=white, inner sep=1pt}}
    
    \Vertex[x=-1,y=1,L={\scriptsize 1}]{V00}
    \Vertex[x=0,y=2,L={\scriptsize 1}]{V00}
    
    \tikzset{VertexStyle/.append style={black,text=white}}

    \Vertex[x=2,y=1,L={\scriptsize 2}]{V00}
    \Vertex[x=1,y=2,L={\scriptsize 2}]{V00}
    
    \tikzset{VertexStyle/.append style={black,text=white}}

    \Vertex[x=4,y=0,L={\scriptsize 3}]{V00}
    \Vertex[x=3,y=1,L={\scriptsize 3}]{V00}
    
    \tikzset{VertexStyle/.append style={black,text=white}}

    \Vertex[x=4,y=-1,L={\scriptsize 4}]{V00}
    \Vertex[x=3,y=-2,L={\scriptsize 4}]{V00}
    
    \tikzset{VertexStyle/.append style={black,text=white}}

    \Vertex[x=3,y=-3,L={\scriptsize 5}]{V00}
    \Vertex[x=2,y=-4,L={\scriptsize 5}]{V00}
    
    \tikzset{VertexStyle/.append style={black,text=white}}

    \Vertex[x=2,y=-5,L={\scriptsize 6}]{V00}
    \Vertex[x=1,y=-6,L={\scriptsize 6}]{V00}
    
    \tikzset{VertexStyle/.append style={black,text=white}}

    \Vertex[x=-1,y=-5,L={\scriptsize 7}]{V00}
    \Vertex[x=0,y=-6,L={\scriptsize 7}]{V00}
    
    \tikzset{VertexStyle/.append style={black,text=white}}

    \Vertex[x=-3,y=-4,L={\scriptsize 8}]{V00}
    \Vertex[x=-2,y=-5,L={\scriptsize 8}]{V00}
    
    \tikzset{VertexStyle/.append style={black,text=white}}

    \Vertex[x=-5,y=-3,L={\scriptsize 9}]{V00}
    \Vertex[x=-4,y=-4,L={\scriptsize 9}]{V00}
    
    \tikzset{VertexStyle/.append style={black,text=white, inner sep=0.1pt}}

    \Vertex[x=-7,y=-2,L={\scriptsize 10}]{V00}
    \Vertex[x=-6,y=-3,L={\scriptsize 10}]{V00}
    
    \tikzset{VertexStyle/.append style={black,text=white}}

    \Vertex[x=-7,y=-1,L={\scriptsize 11}]{V00}
    \Vertex[x=-6,y=0,L={\scriptsize 11}]{V00}
    
    \tikzset{VertexStyle/.append style={black,text=white}}

    \Vertex[x=-6,y=1,L={\scriptsize 12}]{V00}
    \Vertex[x=-5,y=2,L={\scriptsize 12}]{V00}
    
    \tikzset{VertexStyle/.append style={black,text=white}}

    \Vertex[x=-5,y=3,L={\scriptsize 13}]{V00}
    \Vertex[x=-4,y=4,L={\scriptsize 13}]{V00}
    
    \tikzset{VertexStyle/.append style={black,text=white}}

    \Vertex[x=-4,y=5,L={\scriptsize 14}]{V00}
    \Vertex[x=-3,y=6,L={\scriptsize 14}]{V00}
    
    \tikzset{VertexStyle/.append style={black,text=white}}

    \Vertex[x=-3,y=7,L={\scriptsize 15}]{V00}
    \Vertex[x=-2,y=8,L={\scriptsize 15}]{V00}
    
    \tikzset{VertexStyle/.append style={black,text=white}}

    \Vertex[x=0,y=7,L={\scriptsize 16}]{V00}
    \Vertex[x=-1,y=8,L={\scriptsize 16}]{V00}
    
    \tikzset{VertexStyle/.append style={black,text=white}}

    \Vertex[x=2,y=6,L={\scriptsize 17}]{V00}
    \Vertex[x=1,y=7,L={\scriptsize 17}]{V00}
    
    \tikzset{VertexStyle/.append style={black,text=white}}

    \Vertex[x=4,y=5,L={\scriptsize 18}]{V00}
    \Vertex[x=3,y=6,L={\scriptsize 18}]{V00}
    
    \tikzset{VertexStyle/.append style={black,text=white, inner sep=0.1pt}}

    \Vertex[x=6,y=4,L={\scriptsize 19}]{V00}
    \Vertex[x=5,y=5,L={\scriptsize 19}]{V00}
    
    \tikzset{VertexStyle/.append style={black,text=white}}

    \Vertex[x=8,y=3,L={\scriptsize 20}]{V00}
    \Vertex[x=7,y=4,L={\scriptsize 20}]{V00}
    
    \tikzset{VertexStyle/.append style={black,text=white}}

    \Vertex[x=8,y=2,L={\scriptsize 21}]{V00}
    \Vertex[x=7,y=1,L={\scriptsize 21}]{V00}
    
    \tikzset{VertexStyle/.append style={black,text=white}}

    \Vertex[x=7,y=0,L={\scriptsize 22}]{V00}
    \Vertex[x=6,y=-1,L={\scriptsize 22}]{V00}
    
    \tikzset{VertexStyle/.append style={black,text=white}}

    \Vertex[x=6,y=-2,L={\scriptsize 23}]{V00}
    \Vertex[x=5,y=-3,L={\scriptsize 23}]{V00}
    
    \tikzset{VertexStyle/.append style={black,text=white}}

    \Vertex[x=4,y=-3,L={\scriptsize 24}]{V00}

\end{tikzpicture}
\caption{A strategy for containment on the infinite square grid with two firefighters when $d=3$. Recall that burnt vertices are round and orange whereas defended vertices are square and black, except in round zero where they are diamond and star shaped respectively.}
    \label{fig:square_d3_2f}
\end{figure}

\subsection{Square Grid when $\mathbf{d=2}$} \label{subsec:square2}

As previously mentioned, the case of $d=2$ is much more challenging than the case of $d\neq2$, so we need to introduce more complex strategies. In Theorem~\ref{thm:column} we refer to \textit{corralling} the fire to a column, which simply means that the two firefighters are moving in the same direction in two parallel lines and if they continued like this then only vertices between those two lines would burn (see Figure~\ref{fig:column_corral}). We also require that this column of fire is only burning in one direction, rather than in two directions. More formally we say there exist $x_0,y_0,x_1,y_1 \in \mathbb{Z}$ such that $x_0 < x_1$, $y_0 < y_1$, the firefighters are defending $(x_0,y_0 + t)$ and $(x_1,y_1 + t)$ $t$ turns after the corralling has begun\footnote{We assume the column is burning upwards but of course we can always reorient the grid so whichever direction the column is burning is upwards.}, and no new vertices of the form $(x,y)$ with $x \not\in [x_0,x_1]$ and $y > y_1$ or of the form $(x,y)$ with $y < y_0$ will burn if the firefighters continue with this strategy.

\begin{figure}[H]
    \centering
\begin{tikzpicture}[scale=0.5]
    \GraphInit[vstyle=Classic]
    \SetUpVertex[FillColor=white]

    \tikzset{VertexStyle/.append style={minimum size=8pt, inner sep=1pt}}

    \foreach \y in {-6,-5,...,1} {\foreach \x in {-7,-6,...,-3} {\Vertex[x=\x,y=\y,NoLabel=true,]{V\x\y}}}
    \foreach[count =\i, evaluate=\i as \z using int(\i-6)] \y in {-6,-5,...,0} {\foreach \x in {-7,-6,...,-3} {\Edge(V\x\y)(V\x\z)}}
    \foreach[count =\i, evaluate=\i as \z using int(\i-7)] \x in {-7,-6,...,-4} {\foreach \y in {-6,-5,...,1} {\Edge(V\x\y)(V\z\y)}}
    
    \tikzset{VertexStyle/.append style={orange}}
    
    \Vertex[x=-5,y=-5,NoLabel=true]{V00}
    \Vertex[x=-4,y=-5,NoLabel=true]{V00}
    \Vertex[x=-6,y=-5,NoLabel=true]{V00}
    
    \Vertex[x=-5,y=-4,NoLabel=true]{V00}
    \Vertex[x=-4,y=-4,NoLabel=true]{V00}
    \Vertex[x=-6,y=-4,NoLabel=true]{V00}
    
    \Vertex[x=-5,y=-3,NoLabel=true]{V00}
    \Vertex[x=-4,y=-3,NoLabel=true]{V00}
    \Vertex[x=-6,y=-3,NoLabel=true]{V00}
    
    \Vertex[x=-5,y=-2,NoLabel=true]{V00}
    \Vertex[x=-4,y=-2,NoLabel=true]{V00}
    \Vertex[x=-6,y=-2,NoLabel=true]{V00}
    
    \Vertex[x=-5,y=-1,NoLabel=true]{V00}
    \Vertex[x=-4,y=-1,NoLabel=true]{V00}
    \Vertex[x=-6,y=-1,NoLabel=true]{V00}
    
    \Vertex[x=-5,y=0,NoLabel=true]{V00}
    \Vertex[x=-4,y=0,NoLabel=true]{V00}
    \Vertex[x=-6,y=0,NoLabel=true]{V00}
    
    \tikzset{VertexStyle/.append style={rectangle, black}}
    
    \Vertex[x=-5,y=-6,NoLabel=true]{V00}
    \Vertex[x=-4,y=-6,NoLabel=true]{V00}
    \Vertex[x=-6,y=-6,NoLabel=true]{V00}
    
    \Vertex[x=-3,y=-5,NoLabel=true]{V00}
    \Vertex[x=-7,y=-5,NoLabel=true]{V00}
    
    \Vertex[x=-3,y=-4,NoLabel=true]{V00}
    \Vertex[x=-7,y=-4,NoLabel=true]{V00}
    
    \Vertex[x=-3,y=-3,NoLabel=true]{V00}
    \Vertex[x=-7,y=-3,NoLabel=true]{V00}
    
    \Vertex[x=-3,y=-2,NoLabel=true]{V00}
    \Vertex[x=-7,y=-2,NoLabel=true]{V00}
    
    \Vertex[x=-3,y=-1,NoLabel=true]{V00}
    \Vertex[x=-7,y=-1,NoLabel=true]{V00}
    
\end{tikzpicture}
\caption{Two firefighters corralling the fire to a column. Here the fire has just spread and the firefighters (top left and top right) will now each move up one to prevent the fire from spreading out to the side.}
    \label{fig:column_corral}
\end{figure}

\begin{theorem} \label{thm:column}
    Let $d=2$. If two firefighters can corral the fire to a column of finite width in the infinite square grid, then they can contain the fire.
\end{theorem}

\begin{proof}
    The proof of this theorem relies on the fact that two firefighters can contain a fire in the infinite half square grid (the subgraph induced by the vertex set $\{(x,y) \mid y \geq 0\}$) when $d=2$, even if the fire initially breaks out at any finite set of vertices $S$ in the half grid. We in fact prove something stronger, namely that if we instead start the fire at $(0,0)$ and allow it to burn until all of the vertices in $S$ are on fire then the firefighters can still contain the fire. So suppose now that the fire started at $(0,0)$ and after $c$ turns all the vertices of $S$ are on fire. We assume $c$ is even since if it takes an odd number of turns for all of $S$ to be on fire then we can allow the fire to burn for one more turn and proceed as if it took an even number of turns. The firefighters now execute a very simple strategy. First they defend the vertices $(-(c+1),0)$ and $(-(c+1),1)$, and on subsequent turns they each move two vertices upwards to build a wall along the line $x=-(c+1)$. They continue until they defend the vertices $(-(c+1),4c)$ and $(-(c+1),4c+1)$. At this point the largest $y$ value among the burning vertices is $3c$, at the vertex $(0,3c)$. The firefighters now defend the vertices $(-c,4c+1)$ and $(-(c-1),4c+1)$ and move two to the right on each subsequent turn. When the firefighters reach the vertices $(0,4c+1)$ and $(1,4c+1)$ the maximum $y$ value of the fire is $3c + \frac{c}{2}$, which is less than $4c+1$, so the firefighters can continue moving to the right. At this point the maximum $x$ value among the burning vertices is $3c + \frac{c}{2}$, at the vertex $(3c + \frac{c}{2},0)$. When the firefighters reach the vertices $(13c,4c+1)$ and $(13c+1,4c+1)$ the maximum $x$ value of the fire is $10c$. It is clear that at this point if the firefighters turn downwards and create a barrier along the line $x=13c+1$ that they will reach the vertex $(13c+1,0)$ before the fire and thus contain the fire. The initial part of this strategy with $c=2$ is depicted in Figure~\ref{fig:diamond}.

    We now prove the result. Suppose the fire is corralled to a column by two firefighters on the infinite square grid and that $d=2$. The firefighters will continue moving upwards and corralling the fire for $i$ turns. At this point, the firefighter with the larger $x$ coordinate, call this $x_1$, continues as they were while the one with the smaller $x$ value, call this $x_0$, starts moving up two vertices at a time instead of one vertex at a time. Once the second firefighter's $y$ coordinate is at least $x_1-x_0$ larger than the largest $y$ coordinate of any burning vertex, they begin instead moving two vertices to the right until their $x$ coordinate is equal to $x_1$, and they then wait for the first firefighter to reach their position. These two firefighters then continue upwards defending vertices of the form $(x_1,y_1+t)$ and $(x_1,y_1+t+1)$ for some constant $y_1$ until their $y$ values, call them $\bar{y} - 1$ and $\bar{y}$, are at least $x_1-x_0$ larger than the largest $y$ value of any burning vertex. At this point they can turn left and begin defending vertices of the form $(x_1 - t, \bar{y})$ and $(x_0 - t + 1, \bar{y})$. Once their $x$ coordinates are at most $x_0$ we know that the max $y$ value of any burning vertex is bounded above by $\bar{y}$. Now if the firefighters consider the line $x=x_0$ as the boundary of a half grid then we have a finite source fire where our firefighters are already in the starting position we described above. As long as $i$ was large enough the firefighters will be able to surround the fire in this subgraph before the fire can get past the bottom of this wall, and the fire will never be able to get past the wall the firefighters build along the line $y=\bar{y}$.
\end{proof}

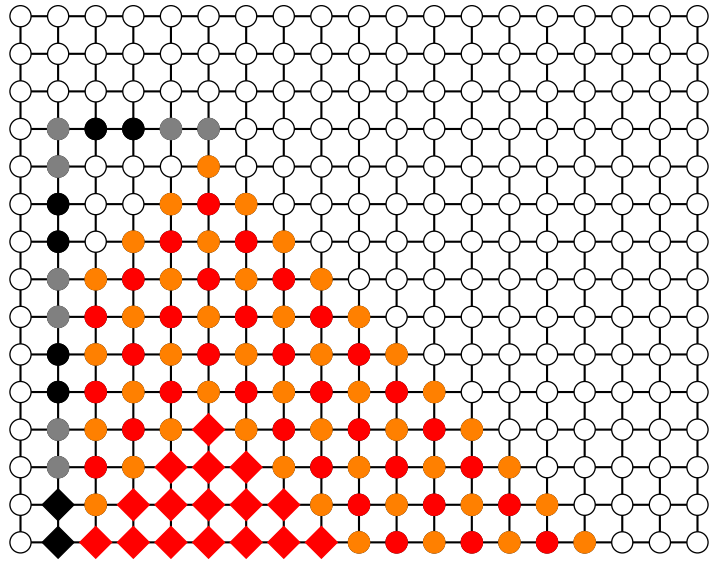
\begin{figure}[H]
    \centering
\begin{tikzpicture}[scale=0.5]
    \SetUpVertex[FillColor=white]

    \tikzset{VertexStyle/.append style={minimum size=8pt, inner sep=1pt}}

    \foreach \y in {-10,-9,...,4} {\foreach \x in {-10,-9,...,8} {\Vertex[x=\x,y=\y,NoLabel=true,]{V\x\y}}}
    \foreach[count =\i, evaluate=\i as \z using int(\i-10)] \y in {-10,-9,...,3} {\foreach \x in {-10,-9,...,8} {\Edge(V\x\y)(V\x\z)}}
    \foreach[count =\i, evaluate=\i as \z using int(\i-10)] \x in {-10,-9,...,7} {\foreach \y in {-10,-9,...,4} {\Edge(V\x\y)(V\z\y)}}
    
    \tikzset{VertexStyle/.append style={orange, diamond, minimum size = 12pt}}
    \tikzset{VertexStyle/.append style={text=black, inner sep=1pt}}
    
    \Vertex[x=-7,y=-10,L={\scriptsize 0}]{V00}
    \Vertex[x=-6,y=-10,L={\scriptsize 0}]{V00}
    \Vertex[x=-6,y=-9,L={\scriptsize 0}]{V00}
    \Vertex[x=-5,y=-10,L={\scriptsize 0}]{V00}
    \Vertex[x=-5,y=-9,L={\scriptsize 0}]{V00}
    \Vertex[x=-5,y=-8,L={\scriptsize 0}]{V00}
    \Vertex[x=-4,y=-10,L={\scriptsize 0}]{V00}
    \Vertex[x=-4,y=-9,L={\scriptsize 0}]{V00}
    \Vertex[x=-3,y=-10,L={\scriptsize 0}]{V00}
    
    \tikzset{VertexStyle/.append style={orange, circle, minimum size = 8pt, text=black, inner sep=1pt}}

    \Vertex[x=-2,y=-10,L={\scriptsize 1}]{V00}
    \Vertex[x=-3,y=-9,L={\scriptsize 1}]{V00}
    \Vertex[x=-4,y=-8,L={\scriptsize 1}]{V00}
    \Vertex[x=-5,y=-7,L={\scriptsize 1}]{V00}
    \Vertex[x=-6,y=-8,L={\scriptsize 1}]{V00}
    \Vertex[x=-7,y=-9,L={\scriptsize 1}]{V00}
    
    \Vertex[x=-7,y=-8,L={\scriptsize 2}]{V00}
    \Vertex[x=-6,y=-7,L={\scriptsize 2}]{V00}
    \Vertex[x=-5,y=-6,L={\scriptsize 2}]{V00}
    \Vertex[x=-4,y=-7,L={\scriptsize 2}]{V00}
    \Vertex[x=-3,y=-8,L={\scriptsize 2}]{V00}
    \Vertex[x=-2,y=-9,L={\scriptsize 2}]{V00}
    \Vertex[x=-1,y=-10,L={\scriptsize 2}]{V00}

    \Vertex[x=-7,y=-7,L={\scriptsize 3}]{V00}
    \Vertex[x=-6,y=-6,L={\scriptsize 3}]{V00}
    \Vertex[x=-5,y=-5,L={\scriptsize 3}]{V00}
    \Vertex[x=-4,y=-6,L={\scriptsize 3}]{V00}
    \Vertex[x=-3,y=-7,L={\scriptsize 3}]{V00}
    \Vertex[x=-2,y=-8,L={\scriptsize 3}]{V00}
    \Vertex[x=-1,y=-9,L={\scriptsize 3}]{V00}
    \Vertex[x=0,y=-10,L={\scriptsize 3}]{V00}

    \Vertex[x=-7,y=-6,L={\scriptsize 4}]{V00}
    \Vertex[x=-6,y=-5,L={\scriptsize 4}]{V00}
    \Vertex[x=-5,y=-4,L={\scriptsize 4}]{V00}
    \Vertex[x=-4,y=-5,L={\scriptsize 4}]{V00}
    \Vertex[x=-3,y=-6,L={\scriptsize 4}]{V00}
    \Vertex[x=-2,y=-7,L={\scriptsize 4}]{V00}
    \Vertex[x=-1,y=-8,L={\scriptsize 4}]{V00}
    \Vertex[x=0,y=-9,L={\scriptsize 4}]{V00}
    \Vertex[x=1,y=-10,L={\scriptsize 4}]{V00}

    \Vertex[x=-7,y=-5,L={\scriptsize 5}]{V00}
    \Vertex[x=-6,y=-4,L={\scriptsize 5}]{V00}
    \Vertex[x=-5,y=-3,L={\scriptsize 5}]{V00}
    \Vertex[x=-4,y=-4,L={\scriptsize 5}]{V00}
    \Vertex[x=-3,y=-5,L={\scriptsize 5}]{V00}
    \Vertex[x=-2,y=-6,L={\scriptsize 5}]{V00}
    \Vertex[x=-1,y=-7,L={\scriptsize 5}]{V00}
    \Vertex[x=0,y=-8,L={\scriptsize 5}]{V00}
    \Vertex[x=1,y=-9,L={\scriptsize 5}]{V00}
    \Vertex[x=2,y=-10,L={\scriptsize 5}]{V00}
    
    \Vertex[x=-7,y=-4,L={\scriptsize 6}]{V00}
    \Vertex[x=-6,y=-3,L={\scriptsize 6}]{V00}
    \Vertex[x=-5,y=-2,L={\scriptsize 6}]{V00}
    \Vertex[x=-4,y=-3,L={\scriptsize 6}]{V00}
    \Vertex[x=-3,y=-4,L={\scriptsize 6}]{V00}
    \Vertex[x=-2,y=-5,L={\scriptsize 6}]{V00}
    \Vertex[x=-1,y=-6,L={\scriptsize 6}]{V00}
    \Vertex[x=0,y=-7,L={\scriptsize 6}]{V00}
    \Vertex[x=1,y=-8,L={\scriptsize 6}]{V00}
    \Vertex[x=2,y=-9,L={\scriptsize 6}]{V00}
    \Vertex[x=3,y=-10,L={\scriptsize 6}]{V00}

    \Vertex[x=-7,y=-3,L={\scriptsize 7}]{V00}
    \Vertex[x=-6,y=-2,L={\scriptsize 7}]{V00}
    \Vertex[x=-5,y=-1,L={\scriptsize 7}]{V00}
    \Vertex[x=-4,y=-2,L={\scriptsize 7}]{V00}
    \Vertex[x=-3,y=-3,L={\scriptsize 7}]{V00}
    \Vertex[x=-2,y=-4,L={\scriptsize 7}]{V00}
    \Vertex[x=-1,y=-5,L={\scriptsize 7}]{V00}
    \Vertex[x=0,y=-6,L={\scriptsize 7}]{V00}
    \Vertex[x=1,y=-7,L={\scriptsize 7}]{V00}
    \Vertex[x=2,y=-8,L={\scriptsize 7}]{V00}
    \Vertex[x=3,y=-9,L={\scriptsize 7}]{V00}
    \Vertex[x=4,y=-10,L={\scriptsize 7}]{V00}

    \tikzset{VertexStyle/.append style={black, star, minimum size = 12pt, text=white, inner sep=1pt}}
    \Vertex[x=-8,y=-10,L={\scriptsize 0}]{V00}
    \Vertex[x=-8,y=-9,L={\scriptsize 0}]{V00}
    \tikzset{VertexStyle/.append style={rectangle, minimum size = 8pt, text=white, inner sep=0.5pt}}
    
    \Vertex[x=-8,y=-8,L={\scriptsize 1}]{V00}
    \Vertex[x=-8,y=-7,L={\scriptsize 1}]{V00}

    \Vertex[x=-8,y=-6,L={\scriptsize 2}]{V00}
    \Vertex[x=-8,y=-5,L={\scriptsize 2}]{V00}

    \Vertex[x=-8,y=-4,L={\scriptsize 3}]{V00}
    \Vertex[x=-8,y=-3,L={\scriptsize 3}]{V00}

    \Vertex[x=-8,y=-2,L={\scriptsize 4}]{V00}
    \Vertex[x=-8,y=-1,L={\scriptsize 4}]{V00}

    \Vertex[x=-8,y=0,L={\scriptsize 5}]{V00}
    \Vertex[x=-8,y=1,L={\scriptsize 5}]{V00}

    \Vertex[x=-7,y=1,L={\scriptsize 6}]{V00}
    \Vertex[x=-6,y=1,L={\scriptsize 6}]{V00}

    \Vertex[x=-5,y=1,L={\scriptsize 7}]{V00}
    \Vertex[x=-4,y=1,L={\scriptsize 7}]{V00}

    \tikzset{EdgeStyle/.append style={line width=5pt}}   
    \tikzset{VertexStyle/.append style={minimum size=1pt, inner sep=0pt}}
    \tikzset{VertexStyle/.append style={white}}

    \Vertex[x=-12,y=-11,NoLabel=true]{A} 
    \Vertex[x=10,y=-11,NoLabel=true]{B}

    \Edge(A)(B)

\end{tikzpicture}
\caption{The beginning of the strategy described in the proof of Theorem~\ref{thm:column} with $c=2$. Note that this is a half grid for which the border is along the bottom of the diagram.}
    \label{fig:diamond}
\end{figure}

    We conjecture that the converse of Theorem~\ref{thm:column} also holds.

\begin{restatable}[]{conjecture}{conja}
\label{conj:column}
    Let $d=2$. Then two firefighters can contain the fire in the infinite square grid if and only if they can contain the fire to a column.
\end{restatable}

Determining $f_2(G_{\square})$ has proven to be a very stubborn problem. We can easily establish that $f_2(G_{\square}) \geq 2$ since two firefighters are optimal in the original game from~\cite{wang2002fire}. With three firefighters the fire is contained in seven turns as illustrated in Figure~\ref{fig:square_d2_3f}, and thus $2 \leq f_2(G_{\square}) \leq 3$. 

\begin{figure}[H]
    \centering
\begin{tikzpicture}[scale=0.5]
    \SetUpVertex[FillColor=white]

    \tikzset{VertexStyle/.append style={minimum size=8pt, inner sep=1pt}}

    \foreach \y in {-2,-1,...,8} {\foreach \x in {-2,-1,...,2} {\Vertex[x=\x,y=\y,NoLabel=true,]{V\x\y}}}
    \foreach[count =\i, evaluate=\i as \z using int(\i-2)] \y in {-2,-1,...,7} {\foreach \x in {-2,-1,...,2} {\Edge(V\x\y)(V\x\z)}}
    \foreach[count =\i, evaluate=\i as \z using int(\i-2)] \x in {-2,-1,...,1} {\foreach \y in {-2,-1,...,8} {\Edge(V\x\y)(V\z\y)}}
    
    \tikzset{VertexStyle/.append style={orange}}
    
    \tikzset{VertexStyle/.append style={diamond, minimum size = 12pt, text=black, inner sep=0.5pt}}
    \Vertex[x=0,y=0,L={\scriptsize 0}]{V00}
    \tikzset{VertexStyle/.append style={circle, minimum size = 8pt, inner sep=1pt}}
    
    \Vertex[x=0,y=1,L={\scriptsize 1}]{V00}
    \Vertex[x=0,y=2,L={\scriptsize 2}]{V00}
    \Vertex[x=0,y=3,L={\scriptsize 3}]{V00}
    \Vertex[x=0,y=4,L={\scriptsize 4}]{V00}
    \Vertex[x=0,y=5,L={\scriptsize 5}]{V00}
    \Vertex[x=0,y=6,L={\scriptsize 6}]{V00}
    
    \tikzset{VertexStyle/.append style={black}}
    \tikzset{VertexStyle/.append style={star, minimum size = 12pt, text=white, inner sep=0.5pt}}
    \Vertex[x=-1,y=0,L={\scriptsize 0}]{V00}
    \Vertex[x=1,y=0,L={\scriptsize 0}]{V00}
    \tikzset{VertexStyle/.append style={rectangle, minimum size = 8pt, inner sep=1pt}}
    
    \Vertex[x=-1,y=1,L={\scriptsize 1}]{V00}
    \Vertex[x=1,y=1,L={\scriptsize 1}]{V00}
    
    \Vertex[x=-1,y=2,L={\scriptsize 2}]{V00}
    \Vertex[x=1,y=2,L={\scriptsize 2}]{V00}
    
    \Vertex[x=-1,y=3,L={\scriptsize 3}]{V00}
    \Vertex[x=1,y=3,L={\scriptsize 3}]{V00}

    \Vertex[x=-1,y=4,L={\scriptsize 4}]{V00}
    \Vertex[x=1,y=4,L={\scriptsize 4}]{V00}

    \Vertex[x=-1,y=5,L={\scriptsize 5}]{V00}
    \Vertex[x=1,y=5,L={\scriptsize 5}]{V00}

    \Vertex[x=-1,y=6,L={\scriptsize 6}]{V00}
    \Vertex[x=1,y=6,L={\scriptsize 6}]{V00}
    
    \tikzset{VertexStyle/.append style={color=brown!55!black}}
    
    \tikzset{VertexStyle/.append style={star, minimum size = 12pt, text=white, inner sep=0.5pt}}
    \Vertex[x=0,y=-1,L={\scriptsize 0}]{V00}
    \tikzset{VertexStyle/.append style={rectangle, minimum size = 8pt, inner sep=1pt}}
    \Vertex[x=-2,y=-1,L={\scriptsize 1}]{V00}
    \Vertex[x=-2,y=1,L={\scriptsize 2}]{V00}
    \Vertex[x=-2,y=3,L={\scriptsize 3}]{V00}
    \Vertex[x=-2,y=5,L={\scriptsize 4}]{V00}
    \Vertex[x=-2,y=7,L={\scriptsize 5}]{V00}
    \Vertex[x=0,y=7,L={\scriptsize 6}]{V00}

\end{tikzpicture}
\caption{A strategy for containment of a single fire on the infinite square grid with three firefighters when $d=2$. The third firefighter is depicted here in brown to highlight the fact that it plays a different role from the other two firefighters.}
    \label{fig:square_d2_3f}
\end{figure}

However, a successful strategy with two firefighters remains elusive. It appears that two firefighters can only restrict the fire to a quarter plane in the case of $d=2$. An approach which has been used to deal with cases where it is unknown whether or not a certain number of firefighters suffices is to give the firefighters some small advantage and show that when using that advantage they can contain the fire. In~\cite{english2021firefighting,gavenciak2014firefighting} the hexagonal grid is considered where there is one firefighter available except on one or two extra turns where there is an extra firefighter. We consider this type of advantage as well as some others over the course of the next two lemmas and the discussion preceding Figure~\ref{fig:sum_dist}.

\begin{lemma} \label{lemma:miss_col}
    Let $d=2$. A fire in the infinite square grid, minus the path $P$ induced by the vertices of the form $(0,t)$ for all $t \in \mathbb{Z}^{-}$, can be contained by two firefighters.
\end{lemma}

\begin{proof}
    First, if the fire starts on a vertex $(\pm 1,t)$ for some $t \in \mathbb{Z}^-$ then the first firefighter starts directly above the fire and the second starts right of the fire if the fire is at $(1,t)$ and left of the fire if the fire is at $(-1,t)$. The firefighters then build a wall downwards through $(\pm 2,t)$ until they are ahead of the fire which is guaranteed because the firefighters are building the wall two vertices at a time and the fire is moving one vertex at a time. Then the firefighters turn in towards the $y$-axis and have thus surrounded the fire.

    Now suppose the fire starts on a vertex that is not of the form $(\pm 1, t)$. Then the firefighters can restrict the fire to a quarter plane within the modified infinite square grid. This quarter grid can always be set up so that one firefighter ($F_1$) is moving in the negative $y$-direction and the other ($F_2$) is moving away from the missing vertices in the $x$-direction. If $F_1$ is between the fire and the missing vertices, then $F_1$ turns and moves towards the missing vertices until reaching the missing vertices. Otherwise it must be that $F_1$ is on the line $x=0$ and so will continue moving in the negative $y$-direction until reaching the missing vertices. From here it is easy to observe that $F_1$ can retrace their steps and then follow the steps of $F_2$ at twice the speed as before and meet up with $F_2$. At this point the firefighters can clearly contain the fire since it is restricted to a half grid at this point and the firefighters are on one side of the fire.
\end{proof}

\begin{figure}[H]
    \centering
\begin{tikzpicture}[scale=0.5]
    \SetUpVertex[FillColor=white]

    \tikzset{VertexStyle/.append style={minimum size=8pt, inner sep=1pt}}

    \foreach \y in {-3,-2,...,4} {\foreach \x in {-1,0,...,6} {
        \ifnum\x=0
            \ifnum\y>-1
                \Vertex[x=\x,y=\y,NoLabel=true,]{V\x\y}
            \fi
        \else
            \Vertex[x=\x,y=\y,NoLabel=true,]{V\x\y} 
        \fi
    }}
    \foreach[count =\i, evaluate=\i as \z using int(\i-3)] \y in {-3,-2,...,3} {\foreach \x in {-1,0,...,6} {
        \ifnum\x=0
            \ifnum\y>-1
                \ifnum\z>-1
                    \Edge(V\x\y)(V\x\z)
                \fi
            \fi
        \else
            \Edge(V\x\y)(V\x\z)
        \fi
    }}
    \foreach[count =\i, evaluate=\i as \z using int(\i-1)] \x in {-1,0,...,5} {\foreach \y in {-3,-2,...,4} {
        \ifnum\x=0
            \ifnum\y>-1
                \Edge(V\x\y)(V\z\y) 
            \else
                \Vertex[x=1,y=1,NoLabel=true]{V11}
            \fi
        \else
            \ifnum\z=0
                \ifnum\y>-1
                    \Edge(V\x\y)(V\z\y) 
                \else
                    \Vertex[x=1,y=1,NoLabel=true]{V11}
                \fi
            \else
                \Edge(V\x\y)(V\z\y)
            \fi
        \fi
    }}
    
    \tikzset{VertexStyle/.append style={orange, diamond, minimum size = 12pt, text=black, inner sep=0.5pt}}    
    
    \Vertex[x=2,y=2,L={\scriptsize 0}]{V00}
    
    \tikzset{VertexStyle/.append style={circle, minimum size = 8pt, inner sep=1pt}}
    
    \Vertex[x=3,y=2,L={\scriptsize 1}]{V00}
    \Vertex[x=2,y=1,L={\scriptsize 1}]{V00}
    
    \Vertex[x=4,y=2,L={\scriptsize 2}]{V00}
    \Vertex[x=3,y=1,L={\scriptsize 2}]{V00}
    \Vertex[x=2,y=0,L={\scriptsize 2}]{V00}
    
    \Vertex[x=5,y=2,L={\scriptsize 3}]{V00}
    \Vertex[x=4,y=1,L={\scriptsize 3}]{V00}
    \Vertex[x=3,y=0,L={\scriptsize 3}]{V00}
    \Vertex[x=2,y=-1,L={\scriptsize 3}]{V00}
    
    \tikzset{VertexStyle/.append style={black, star, minimum size = 12pt, text=white, inner sep=0.5pt}}    

    \Vertex[x=2,y=3,L={\scriptsize 0}]{V22}
    \Vertex[x=1,y=2,L={\small \sfrac{0}{3}}]{V43}

    \tikzset{VertexStyle/.append style={rectangle, minimum size = 8pt, inner sep=1pt}}

    \Vertex[x=3,y=3,L={\scriptsize 1}]{V22}
    \Vertex[x=1,y=1,L={\scriptsize 1}]{V43}

    \Vertex[x=4,y=3,L={\scriptsize 2}]{V22}
    \Vertex[x=1,y=0,L={\scriptsize 2}]{V43}
    
    \Vertex[x=5,y=3,L={\scriptsize 3}]{V22}

\end{tikzpicture}
\caption{The beginning of the firefighters' strategy for the grid with the path $P$ removed as described in Lemma~\ref{lemma:miss_col}. The vertex with a fraction as the label represents the fact that it was defended both on turn zero and on turn three.}
    \label{fig:11}
\end{figure}    

    It is difficult to remove a simpler infinite connected structure from the grid, so we can see that containing the fire here with two firefighters is very close to being able to contain the fire in the original grid. Based on this it appears that two firefighters are nearly sufficient in the infinite square grid when $d=2$. This notion of being nearly sufficient is further strengthened by Lemma~\ref{lemma:oneextra}.

\begin{lemma} \label{lemma:oneextra}
    In the distance-restricted firefighter game with two firefighters on the infinite square grid with $d=2$, if there is one extra firefighter on an unspecified turn, then the firefighters can contain the fire.
\end{lemma}

\begin{proof}
    Without loss of generality, suppose the fire starts at $(1,1)$. The firefighters ($F_1,F_2$) begin by defending vertices of the form $(0,t)$ and $(t,0)$ respectively at time $t$. Suppose the extra firefighter is available at time $T$. The extra firefighter will defend the vertex $(T,0)$ while $F_1$ defends $(0,T)$ and $F_2$ defends $(T+1,0)$. The vertex $(T+1,1)$ will then burn and the firefighters will then defend the vertices $(0,T+1)$ and $(T+2,1)$ respectively. The fire is now corralled to a column and so the firefighters can contain the fire by Theorem~\ref{thm:column}.
\end{proof}

We can also allow for a slightly different variation on our rules to further solidify that two firefighters are almost sufficient in the case of $d=2$. If instead we require that the two firefighters' distance moved on each turn sums to at most four, then the firefighters can contain the fire given that they can move through the fire on exactly one turn (see Figure~\ref{fig:sum_dist}). In general we define this game with $k$ firefighters by saying that the sum of their distances on any turn is at most $d \cdot k$ and we refer to this modified game as \textit{sum-distance} firefighting.

\begin{figure}[H]
    \centering
\begin{tikzpicture}[scale=0.5]
    \SetUpVertex[FillColor=white]

    \tikzset{VertexStyle/.append style={minimum size=8pt, inner sep=1pt}}
    
    \foreach \y in {-2,-1,...,4} {\foreach \x in {-4,-3,...,2} {\Vertex[x=\x,y=\y,NoLabel=true,]{V\x\y}}}
    \foreach[count =\i, evaluate=\i as \z using int(\i-2)] \y in {-2,-1,...,3} {\foreach \x in {-4,-3,...,2} {\Edge(V\x\y)(V\x\z)}}
    \foreach[count =\i, evaluate=\i as \z using int(\i-4)] \x in {-4,-3,...,1} {\foreach \y in {-2,-1,...,4} {\Edge(V\x\y)(V\z\y)}}

    \tikzset{VertexStyle/.append style={orange, diamond, minimum size=12pt, text=black, inner sep=0.5pt}}
    
    \Vertex[x=0,y=0,L={\scriptsize 0}]{V22}

    \tikzset{VertexStyle/.append style={orange, circle, minimum size=8pt, inner sep=1pt, text=black}}
    
    \Vertex[x=0,y=1,L={\scriptsize 1}]{V22}
    \Vertex[x=1,y=0,L={\scriptsize 1}]{V22}
    
    \tikzset{VertexStyle/.append style={orange, text=black}}
    
    \Vertex[x=0,y=2,L={\scriptsize 2}]{V22}
    \Vertex[x=-1,y=1,L={\scriptsize 2}]{V22}
    \Vertex[x=1,y=1,L={\scriptsize 2}]{V22}
    
    \tikzset{VertexStyle/.append style={orange, text=black}}
    
    \Vertex[x=0,y=3,L={\scriptsize 3}]{V22}
    \Vertex[x=-1,y=2,L={\scriptsize 3}]{V22}
    \Vertex[x=1,y=2,L={\scriptsize 3}]{V22}
    \Vertex[x=-2,y=1,L={\scriptsize 3}]{V22}
    
    \tikzset{VertexStyle/.append style={orange, text=black}}
    
    \Vertex[x=0,y=4,L={\scriptsize 4}]{V22}
    \Vertex[x=-1,y=3,L={\scriptsize 4}]{V22}
    \Vertex[x=1,y=3,L={\scriptsize 4}]{V22}
    \Vertex[x=-2,y=2,L={\scriptsize 4}]{V22}
    \Vertex[x=-3,y=1,L={\scriptsize 4}]{V22}
    \Vertex[x=-2,y=0,L={\scriptsize 4}]{V22}

    \tikzset{VertexStyle/.append style={black, star, minimum size=12pt, text=white, inner sep=0.5pt}}

    \Vertex[x=0,y=-1,L={\scriptsize 0}]{V22}
    \Vertex[x=-1,y=0,L={\scriptsize 0}]{V22}
    
    \tikzset{VertexStyle/.append style={black, rectangle, minimum size=8pt, text=white, inner sep=1pt}}

    \Vertex[x=1,y=-1,L={\scriptsize 1}]{V22}
    \Vertex[x=2,y=0,L={\scriptsize 1}]{V22}

    \tikzset{VertexStyle/.append style={black, text=white}}

    \Vertex[x=-2,y=-1,L={\scriptsize 2}]{V22}
    \Vertex[x=2,y=1,L={\scriptsize 2}]{V22}

    \tikzset{VertexStyle/.append style={black, text=white}}

    \Vertex[x=-3,y=0,L={\scriptsize 3}]{V22}
    \Vertex[x=2,y=2,L={\scriptsize 3}]{V22}

    \tikzset{VertexStyle/.append style={black, text=white}}

    \Vertex[x=-4,y=1,L={\scriptsize 4}]{V22}
    \Vertex[x=2,y=3,L={\scriptsize 4}]{V22}

\end{tikzpicture}
\caption{The initial part of the strategy for containing the fire in the sum-distance game with two firefighters when $d=2$. Observe that the firefighters have the fire corralled to a column at the last turn in the figure, so by Theorem~\ref{thm:column} the firefighters can contain the fire.}
    \label{fig:sum_dist}
\end{figure}    

From Lemma~\ref{lemma:miss_col}, Lemma~\ref{lemma:oneextra}, and Figure~\ref{fig:sum_dist} it is clear that very slight deviations from the problem of containing the fire with two firefighters when $d=2$ results in the firefighters being able to contain the fire. This is despite the fact that it appears to be impossible for two firefighters to contain the fire on the infinite square grid when $d=2$, which we formalize in Conjecture~\ref{conj:two_insuff}.

\begin{restatable}[]{conjecture}{conjb}
\label{conj:two_insuff}
    Two firefighters do not suffice to contain the fire on the infinite square grid when $d=2$.
\end{restatable}

We further expand upon Conjecture~\ref{conj:two_insuff} below in Conjecture~\ref{conj:three_quart} using a notion from~\cite{gavenciak2014firefighting} with different notation. The notion of saving some portion $\rho \in [0,1]$ of the vertices with some predetermined strategy is defined as $\liminf_{n\to\infty} \frac{|B_n|}{|D_n|} = \rho$. Here $D_n$ is the set of vertices at distance $n$ or less from where the fire broke out and $B_n$ is the set of vertices at distance $n$ or less from where the fire broke out that will eventually burn. In this case we say the firefighters can save $\rho$ of the vertices. 

\begin{restatable}[]{conjecture}{conjc}
\label{conj:three_quart}
    Two firefighters cannot save $\rho$ of the vertices on the infinite square grid when $d=2$ for any $\rho > \frac34$.
\end{restatable}

\section{Strong Grid} \label{sec:strong}

One of the main results from~\cite{days2019firefighter} shows that eight firefighters are necessary on the infinite strong grid when $d=1$ (i.e. $f_1(G_{\boxtimes}) = 8$), and the strategy illustrated in Figure~\ref{fig:strong_d2_4f} shows that $f_2(G_{\boxtimes}) \leq 4$. It is also known that in the original game four firefighters is the minimum number of firefighters needed to contain a fire on the infinite strong grid~\cite{messinger2005firefighting}. Thus the firefighters cannot contain the fire with three firefighters at any distance. If the distance is increased, the same strategy can be used to contain the fire, so all distances greater than two also require four firefighters. 

\begin{lemma} \label{lemma:strong4ffd2}
    Four firefighters are necessary and sufficient to contain the fire on the infinite strong grid when $d=2$.
\end{lemma}

\begin{proof}
    Refer to Figure~\ref{fig:strong_d2_4f} and Appendix~\ref{app:strong} to see that four firefighters are sufficient. The leftmost firefighter moves across the top of the fire while the other three spiral counterclockwise. To confirm that four firefighters are necessary, see Theorem $22$ from~\cite{messinger2005firefighting} which states that three firefighters do not suffice in the original game and therefore they do not suffice for the distance-restricted game as well. Thus four firefighters are both necessary and sufficient to contain the fire on the infinite strong grid when $d=2$.
\end{proof}

\begin{figure}[H]
    \centering
\begin{tikzpicture}[scale=0.5]
    \SetUpVertex[FillColor=white]

    \tikzset{VertexStyle/.append style={minimum size=8pt, inner sep=1pt}}


    \foreach \y in {0,1,...,8} {\foreach \x in {0,1,...,13} {\Vertex[x=\x,y=\y,NoLabel=true,]{V\x\y}}}
    

    \foreach[count =\i, evaluate=\i as \z using int(\i)] \y in {0,1,...,7} {\foreach \x in {0,1,...,13} {\Edge(V\x\y)(V\x\z)}}
    \foreach[count =\i, evaluate=\i as \z using int(\i)] \x in {0,1,...,12} {\foreach \y in {0,1,...,8} {\Edge(V\x\y)(V\z\y)}}
    

    \foreach[count =\i, evaluate=\i as \z using int(\i)] \y in {0,1,...,7} {\foreach[count=\j, evaluate=\j as \u using int(\j)] \x in {0,1,...,12} {\Edge(V\x\y)(V\u\z)}}
    

    \foreach[count =\i, evaluate=\i as \z using int(\i)] \x in {0,1,...,12} {\foreach[count=\j, evaluate=\j as \u using int(\j-1)] \y in {1,2,...,8} {\Edge(V\x\y)(V\z\u)}}

    \tikzset{VertexStyle/.append style={orange}}
    
    \tikzset{VertexStyle/.append style={diamond, minimum size = 12pt, text=black, inner sep=0.5pt}}
    \Vertex[x=11,y=7,L={\scriptsize 0}]{V00}
    \tikzset{VertexStyle/.append style={circle, minimum size = 8pt, inner sep=1pt}}

    \Vertex[x=10,y=7,L={\scriptsize 1}]{V00}
    \Vertex[x=10,y=6,L={\scriptsize 1}]{V00}
    \Vertex[x=11,y=6,L={\scriptsize 1}]{V00}
    \Vertex[x=12,y=6,L={\scriptsize 1}]{V00}

    \Vertex[x=9,y=7,L={\scriptsize 2}]{V00}
    \Vertex[x=9,y=6,L={\scriptsize 2}]{V00}
    \Vertex[x=9,y=5,L={\scriptsize 2}]{V00}
    \Vertex[x=10,y=5,L={\scriptsize 2}]{V00}
    \Vertex[x=11,y=5,L={\scriptsize 2}]{V00}
    \Vertex[x=12,y=5,L={\scriptsize 2}]{V00}

    \Vertex[x=8,y=7,L={\scriptsize 3}]{V00}
    \Vertex[x=8,y=6,L={\scriptsize 3}]{V00}
    \Vertex[x=8,y=5,L={\scriptsize 3}]{V00}
    \Vertex[x=8,y=4,L={\scriptsize 3}]{V00}
    \Vertex[x=9,y=4,L={\scriptsize 3}]{V00}
    \Vertex[x=10,y=4,L={\scriptsize 3}]{V00}
    \Vertex[x=11,y=4,L={\scriptsize 3}]{V00}

    \Vertex[x=7,y=7,L={\scriptsize 4}]{V00}
    \Vertex[x=7,y=6,L={\scriptsize 4}]{V00}
    \Vertex[x=7,y=5,L={\scriptsize 4}]{V00}
    \Vertex[x=7,y=4,L={\scriptsize 4}]{V00}
    \Vertex[x=7,y=3,L={\scriptsize 4}]{V00}
    \Vertex[x=8,y=3,L={\scriptsize 4}]{V00}
    \Vertex[x=9,y=3,L={\scriptsize 4}]{V00}

    \Vertex[x=6,y=7,L={\scriptsize 5}]{V00}
    \Vertex[x=6,y=6,L={\scriptsize 5}]{V00}
    \Vertex[x=6,y=5,L={\scriptsize 5}]{V00}
    \Vertex[x=6,y=4,L={\scriptsize 5}]{V00}
    \Vertex[x=6,y=3,L={\scriptsize 5}]{V00}
    \Vertex[x=6,y=2,L={\scriptsize 5}]{V00}
    \Vertex[x=7,y=2,L={\scriptsize 5}]{V00}

    \Vertex[x=5,y=7,L={\scriptsize 6}]{V00}
    \Vertex[x=5,y=6,L={\scriptsize 6}]{V00}
    \Vertex[x=5,y=5,L={\scriptsize 6}]{V00}
    \Vertex[x=5,y=4,L={\scriptsize 6}]{V00}
    \Vertex[x=5,y=3,L={\scriptsize 6}]{V00}
    \Vertex[x=5,y=2,L={\scriptsize 6}]{V00}
    \Vertex[x=5,y=1,L={\scriptsize 6}]{V00}

    \Vertex[x=4,y=7,L={\scriptsize 7}]{V00}
    \Vertex[x=4,y=6,L={\scriptsize 7}]{V00}
    \Vertex[x=4,y=5,L={\scriptsize 7}]{V00}
    \Vertex[x=4,y=4,L={\scriptsize 7}]{V00}
    \Vertex[x=4,y=3,L={\scriptsize 7}]{V00}
    \Vertex[x=4,y=2,L={\scriptsize 7}]{V00}
    \Vertex[x=4,y=1,L={\scriptsize 7}]{V00}

    \Vertex[x=3,y=7,L={\scriptsize 8}]{V00}
    \Vertex[x=3,y=6,L={\scriptsize 8}]{V00}
    \Vertex[x=3,y=5,L={\scriptsize 8}]{V00}
    \Vertex[x=3,y=4,L={\scriptsize 8}]{V00}
    \Vertex[x=3,y=3,L={\scriptsize 8}]{V00}
    \Vertex[x=3,y=2,L={\scriptsize 8}]{V00}

    \Vertex[x=2,y=7,L={\scriptsize 9}]{V00}
    \Vertex[x=2,y=6,L={\scriptsize 9}]{V00}
    \Vertex[x=2,y=5,L={\scriptsize 9}]{V00}
    \Vertex[x=2,y=4,L={\scriptsize 9}]{V00}

    \tikzset{VertexStyle/.append style={inner sep=0.1pt}}
    
    \Vertex[x=1,y=7,L={\scriptsize 10}]{V00}
    \Vertex[x=1,y=6,L={\scriptsize 10}]{V00}

    \tikzset{VertexStyle/.append style={black, star, minimum size = 12pt, text=white, inner sep=0.5pt}}
    
    \Vertex[x=10,y=8,L={\scriptsize 0}]{V00}
    
    \Vertex[x=12,y=8,L={\scriptsize 0}]{V00}
    \Vertex[x=11,y=8,L={\scriptsize 0}]{V00}
    \Vertex[x=12,y=7,L={\scriptsize 0}]{V00}
    
    \tikzset{VertexStyle/.append style={rectangle, minimum size = 8pt, inner sep=1pt}}
    
    \Vertex[x=9,y=8,L={\scriptsize 1}]{V00}
    
    \Vertex[x=13,y=7,L={\scriptsize 1}]{V00}
    \Vertex[x=13,y=6,L={\scriptsize 1}]{V00}
    \Vertex[x=13,y=5,L={\small \sfrac{1}{2}}]{V00} 

    \Vertex[x=8,y=8,L={\scriptsize 2}]{V00}
    
    \Vertex[x=13,y=4,L={\scriptsize 2}]{V00}
    \Vertex[x=12,y=4,L={\scriptsize 2}]{V00}

    \Vertex[x=7,y=8,L={\scriptsize 3}]{V00}

    \Vertex[x=12,y=3,L={\scriptsize 3}]{V00}
    \Vertex[x=11,y=3,L={\scriptsize 3}]{V00}
    \Vertex[x=10,y=3,L={\scriptsize 3}]{V00}

    \Vertex[x=6,y=8,L={\scriptsize 4}]{V00}

    \Vertex[x=10,y=2,L={\scriptsize 4}]{V00}
    \Vertex[x=9,y=2,L={\scriptsize 4}]{V00}
    \Vertex[x=8,y=2,L={\scriptsize 4}]{V00}

    \Vertex[x=5,y=8,L={\scriptsize 5}]{V00}

    \Vertex[x=8,y=1,L={\scriptsize 5}]{V00}
    \Vertex[x=7,y=1,L={\scriptsize 5}]{V00}
    \Vertex[x=6,y=1,L={\scriptsize 5}]{V00}

    \Vertex[x=4,y=8,L={\scriptsize 6}]{V00}

    \Vertex[x=6,y=0,L={\scriptsize 6}]{V00}
    \Vertex[x=5,y=0,L={\scriptsize 6}]{V00}
    \Vertex[x=4,y=0,L={\small \sfrac{6}{7}}]{V00} 

    \Vertex[x=3,y=8,L={\scriptsize 7}]{V00}

    \Vertex[x=3,y=0,L={\scriptsize 7}]{V00}
    \Vertex[x=3,y=1,L={\scriptsize 7}]{V00}

    \Vertex[x=2,y=8,L={\scriptsize 8}]{V00}

    \Vertex[x=2,y=1,L={\scriptsize 8}]{V00}
    \Vertex[x=2,y=2,L={\scriptsize 8}]{V00}
    \Vertex[x=2,y=3,L={\scriptsize 8}]{V00}

    \Vertex[x=1,y=8,L={\scriptsize 9}]{V00}

    \Vertex[x=1,y=3,L={\scriptsize 9}]{V00}
    \Vertex[x=1,y=4,L={\scriptsize 9}]{V00}
    \Vertex[x=1,y=5,L={\scriptsize 9}]{V00}

    \Vertex[x=0,y=8,L={\scriptsize 10}]{V00}

    \Vertex[x=0,y=5,L={\scriptsize 10}]{V00}
    \Vertex[x=0,y=6,L={\scriptsize 10}]{V00}
    \Vertex[x=0,y=7,L={\scriptsize 10}]{V00}

\end{tikzpicture}
\caption{A strategy for containment with four firefighters on the infinite strong grid when $d=2$. One firefighter moves across the top while the other three spiral around from the other side. The two vertices with fractions as labels represent the fact that these vertices were defended at both turns one and two and turns six and seven respectively.}
    \label{fig:strong_d2_4f}
\end{figure}

In terms of average firefighting, note that a strategy is given for the infinite strong grid with $3+\frac{1}{T}$ firefighters for any $T \in \mathbb{Z}^+$ in~\cite{messinger2008average}. That is to say, there is a strategy where whenever the turn number is $0$ modulo $T$ there are four firefighters available, and the rest of the time there are three firefighters available. This strategy is almost valid for the $d=2$ game and only needs slight modification to satisfy the restrictions. 

    The initial part of the strategy has the firefighters contain the fire to a quarter plane, but then has one of the firefighters jump from being on the right side to being on the left side. We simply modify this so that instead of maintaining the right wall and moving around the fire from top to bottom on the left, we maintain the top wall and move around the fire from right to left along the bottom. The first four moves of this are drawn below in Figure~\ref{fig:strong_d2_3ef} for the case of $T=2$.

\begin{figure}[H]
    \centering
\begin{tikzpicture}[scale=0.5]
    \SetUpVertex[FillColor=white]

    \tikzset{VertexStyle/.append style={minimum size=8pt, inner sep=1pt}}


    \foreach \y in {2,3,...,8} {\foreach \x in {7,8,...,14} {\Vertex[x=\x,y=\y,NoLabel=true,]{V\x\y}}}
    

    \foreach[count =\i, evaluate=\i as \z using int(\i+2)] \y in {2,3,...,7} {\foreach \x in {7,8,...,14} {\Edge(V\x\y)(V\x\z)}}
    \foreach[count =\i, evaluate=\i as \z using int(\i+7)] \x in {7,8,...,13} {\foreach \y in {2,3,...,8} {\Edge(V\x\y)(V\z\y)}}
    

    \foreach[count =\i, evaluate=\i as \z using int(\i+2)] \y in {2,3,...,7} {\foreach[count=\j, evaluate=\j as \u using int(\j + 7)] \x in {7,8,...,13} {\Edge(V\x\y)(V\u\z)}}
    

    \foreach[count =\i, evaluate=\i as \z using int(\i+7)] \x in {7,8,...,13} {\foreach[count=\j, evaluate=\j as \u using int(\j+1)] \y in {3,4,...,8} {\Edge(V\x\y)(V\z\u)}}

    \tikzset{VertexStyle/.append style={orange}}
    
    \tikzset{VertexStyle/.append style={diamond, minimum size = 12pt, text=black, inner sep=0.5pt}}
    \Vertex[x=11,y=7,L={\scriptsize 0}]{V00}
    \tikzset{VertexStyle/.append style={circle, minimum size = 8pt, inner sep=1pt}}

    \Vertex[x=10,y=7,L={\scriptsize 1}]{V00}
    \Vertex[x=10,y=6,L={\scriptsize 1}]{V00}
    \Vertex[x=11,y=6,L={\scriptsize 1}]{V00}
    \Vertex[x=12,y=6,L={\scriptsize 1}]{V00}

    \Vertex[x=9,y=7,L={\scriptsize 2}]{V00}
    \Vertex[x=9,y=6,L={\scriptsize 2}]{V00}
    \Vertex[x=9,y=5,L={\scriptsize 2}]{V00}
    \Vertex[x=10,y=5,L={\scriptsize 2}]{V00}
    \Vertex[x=11,y=5,L={\scriptsize 2}]{V00}
    \Vertex[x=12,y=5,L={\scriptsize 2}]{V00}
    \Vertex[x=13,y=5,L={\scriptsize 2}]{V00}

    \Vertex[x=8,y=7,L={\scriptsize 3}]{V00}
    \Vertex[x=8,y=6,L={\scriptsize 3}]{V00}
    \Vertex[x=8,y=5,L={\scriptsize 3}]{V00}
    \Vertex[x=8,y=4,L={\scriptsize 3}]{V00}
    \Vertex[x=9,y=4,L={\scriptsize 3}]{V00}
    \Vertex[x=10,y=4,L={\scriptsize 3}]{V00}
    \Vertex[x=11,y=4,L={\scriptsize 3}]{V00}
    \Vertex[x=12,y=4,L={\scriptsize 3}]{V00}
    \Vertex[x=13,y=4,L={\scriptsize 3}]{V00}

    \tikzset{VertexStyle/.append style={black, star, minimum size = 12pt, text=white, inner sep=0.5pt}}
    
    \Vertex[x=10,y=8,L={\scriptsize 0}]{V00}
    
    \Vertex[x=12,y=8,L={\scriptsize 0}]{V00}
    \Vertex[x=11,y=8,L={\scriptsize 0}]{V00}
    \Vertex[x=12,y=7,L={\scriptsize 0}]{V00}
    
    \tikzset{VertexStyle/.append style={rectangle, minimum size = 8pt, inner sep=1pt}}

    \Vertex[x=9,y=8,L={\scriptsize 1}]{V00}
    
    \Vertex[x=13,y=7,L={\scriptsize 1}]{V00}
    \Vertex[x=13,y=6,L={\scriptsize 1}]{V00}

    \Vertex[x=8,y=8,L={\scriptsize 2}]{V00}
    
    \Vertex[x=14,y=6,L={\scriptsize 2}]{V00}
    \Vertex[x=14,y=5,L={\scriptsize 2}]{V00}
    \Vertex[x=14,y=4,L={\scriptsize 2}]{V00}

    \Vertex[x=7,y=8,L={\scriptsize 3}]{V00}
    
    \Vertex[x=13,y=3,L={\scriptsize 3}]{V00}
    \Vertex[x=14,y=3,L={\scriptsize 3}]{V00}

\end{tikzpicture}
    \caption{The start of a strategy for containment with $3 + \frac{1}{T}$ firefighters on the infinite strong grid when $d=2$ and $T=2$.}
    \label{fig:strong_d2_3ef}
\end{figure}
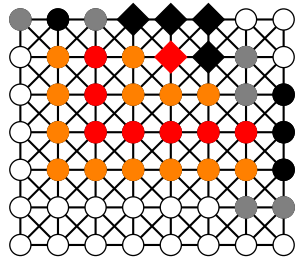

\section{Hexagonal Grid} \label{sec:hex}

The infinite hexagonal grid represents an interesting challenge as the minimum number of firefighters required in the original firefighting game is the subject of Messinger's Conjecture~\cite{messinger2005firefighting}. Messinger's Conjecture states that on the infinite hexagonal grid ($G_{\hexagon}$), one firefighter is insufficient to contain the fire. It was shown in~\cite{gavenciak2014firefighting} that a fire can be contained on the infinite hexagonal grid if there is always one firefighter available except on turns $t_1,t_2$ where there is one additional firefighter. The bound was further improved to only require one extra firefighter on a single turn $t_1$ in~\cite{english2021firefighting}. This presents a new challenge for us as unlike previous cases we do not inherit a good lower bound from the original game.

    The aforementioned challenge is exacerbated in this case due to the fact that the strategies given in~\cite{english2021firefighting} and~\cite{gavenciak2014firefighting} have the firefighter moving further and further as they wait for the extra firefighter to show up. That is to say, if we play the distance-restricted firefighter game on the infinite hexagonal grid with a single firefighter and any finite value of $d$, then there exists a value of $t_1$ (or a pair of values $t_1$ and $t_2$) where if the extra firefighter shows up after time $t_1$ (or after times $t_1$ and $t_2$) then the strategy becomes invalid as the firefighter will have to move a distance greater than $d$. As such there may be no strategy for any finite $d$ where a single extra firefighter on a single turn results in the fire being contained.

    For the case of $d=1$ observe in Figure~\ref{fig:hex_paths} that every vertex has three paths that can be drawn out in a similar way as the four paths from the case of $d=1$ on the infinite square grid. That is to say, a firefighter defending one of these paths cannot defend either of the other two paths. We formalize this by saying that if the vertices are numbered (starting with the $0^{th}$ vertex as the vertex shared by the paths), then the $i^{th}$ vertices in any pair of these paths are at a minimum distance of $i$ from each other. Thus when $d=1$ three firefighters are required and clearly three firefighters suffice.

\begin{lemma}
    Three firefighters are necessary and sufficient to contain the fire on the infinite hexagonal grid when $d=1$.
\end{lemma}

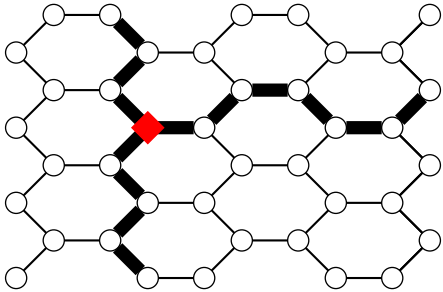
\begin{figure}[H]
    \centering
\begin{tikzpicture}[scale=0.5]
    \GraphInit[vstyle=Classic]
    \SetUpVertex[FillColor=white]

    \tikzset{VertexStyle/.append style={minimum size=8pt, inner sep=1pt}}

	\Vertex[x=0.000000,y=3.000000,NoLabel=true]{V0_3}
	\Vertex[x=0.000000,y=1.000000,NoLabel=true]{V0_1}
	\Vertex[x=0.000000,y=-1.000000,NoLabel=true]{V0_-1}
	\Vertex[x=0.000000,y=-3.000000,NoLabel=true]{V0_-3}
	\Vertex[x=1.000000,y=4.000000,NoLabel=true]{V1_4}
	\Vertex[x=1.000000,y=2.000000,NoLabel=true]{V1_2}
	\Vertex[x=1.000000,y=0.000000,NoLabel=true]{V1_0}
	\Vertex[x=1.000000,y=-2.000000,NoLabel=true]{V1_-2}
	\Vertex[x=-1.500000,y=3.000000,NoLabel=true]{V-2_3}
	\Vertex[x=2.500000,y=4.000000,NoLabel=true]{V2_4}
	\Vertex[x=-1.500000,y=1.000000,NoLabel=true]{V-2_1}
	\Vertex[x=2.500000,y=2.000000,NoLabel=true]{V2_2}
	\Vertex[x=-1.500000,y=-1.000000,NoLabel=true]{V-2_-1}
	\Vertex[x=2.500000,y=0.000000,NoLabel=true]{V2_0}
	\Vertex[x=-1.500000,y=-3.000000,NoLabel=true]{V-2_-3}
	\Vertex[x=2.500000,y=-2.000000,NoLabel=true]{V2_-2}
	\Vertex[x=-2.500000,y=4.000000,NoLabel=true]{V-2_4}
	\Vertex[x=3.500000,y=3.000000,NoLabel=true]{V4_3}
	\Vertex[x=-2.500000,y=2.000000,NoLabel=true]{V-2_2}
	\Vertex[x=3.500000,y=1.000000,NoLabel=true]{V4_1}
	\Vertex[x=-2.500000,y=0.000000,NoLabel=true]{V-2_0}
	\Vertex[x=3.500000,y=-1.000000,NoLabel=true]{V4_-1}
	\Vertex[x=-2.500000,y=-2.000000,NoLabel=true]{V-2_-2}
	\Vertex[x=3.500000,y=-3.000000,NoLabel=true]{V4_-3}
	\Vertex[x=-4.000000,y=4.000000,NoLabel=true]{V-4_4}
	\Vertex[x=5.000000,y=3.000000,NoLabel=true]{V5_3}
	\Vertex[x=-4.000000,y=2.000000,NoLabel=true]{V-4_2}
	\Vertex[x=5.000000,y=1.000000,NoLabel=true]{V5_1}
	\Vertex[x=-4.000000,y=0.000000,NoLabel=true]{V-4_0}
	\Vertex[x=5.000000,y=-1.000000,NoLabel=true]{V5_-1}
	\Vertex[x=-4.000000,y=-2.000000,NoLabel=true]{V-4_-2}
	\Vertex[x=5.000000,y=-3.000000,NoLabel=true]{V5_-3}
	\Vertex[x=-5.000000,y=3.000000,NoLabel=true]{V-5_3}
	\Vertex[x=6.000000,y=4.000000,NoLabel=true]{V6_4}
	\Vertex[x=-5.000000,y=1.000000,NoLabel=true]{V-5_1}
	\Vertex[x=6.000000,y=2.000000,NoLabel=true]{V6_2}
	\Vertex[x=-5.000000,y=-1.000000,NoLabel=true]{V-5_-1}
	\Vertex[x=6.000000,y=0.000000,NoLabel=true]{V6_0}
	\Vertex[x=-5.000000,y=-3.000000,NoLabel=true]{V-5_-3}
	\Vertex[x=6.000000,y=-2.000000,NoLabel=true]{V6_-2}
	\Edge(V0_3)(V-2_3)
	\Edge(V0_3)(V1_4)
	\Edge(V0_3)(V1_2)
	\Edge(V0_1)(V-2_1)
	\Edge(V0_1)(V1_2)
	\Edge(V0_1)(V1_0)
	\Edge(V0_-1)(V-2_-1)
	\Edge(V0_-1)(V1_0)
	\Edge(V0_-1)(V1_-2)
	\Edge(V0_-3)(V-2_-3)
	\Edge(V0_-3)(V1_-2)
	\Edge(V1_4)(V2_4)
	\Edge(V1_2)(V2_2)
	\Edge(V1_0)(V2_0)
	\Edge(V1_-2)(V2_-2)
	\Edge(V-2_3)(V-2_4)
	\Edge(V-2_3)(V-2_2)
	\Edge(V2_4)(V4_3)
	\Edge(V-2_1)(V-2_2)
	\Edge(V-2_1)(V-2_0)
	\Edge(V2_2)(V4_3)
	\Edge(V2_2)(V4_1)
	\Edge(V-2_-1)(V-2_0)
	\Edge(V-2_-1)(V-2_-2)
	\Edge(V2_0)(V4_1)
	\Edge(V2_0)(V4_-1)
	\Edge(V-2_-3)(V-2_-2)
	\Edge(V2_-2)(V4_-1)
	\Edge(V2_-2)(V4_-3)
	\Edge(V-2_4)(V-4_4)
	\Edge(V4_3)(V5_3)
	\Edge(V-2_2)(V-4_2)
	\Edge(V4_1)(V5_1)
	\Edge(V-2_0)(V-4_0)
	\Edge(V4_-1)(V5_-1)
	\Edge(V-2_-2)(V-4_-2)
	\Edge(V4_-3)(V5_-3)
	\Edge(V-4_4)(V-5_3)
	\Edge(V5_3)(V6_4)
	\Edge(V5_3)(V6_2)
	\Edge(V-4_2)(V-5_3)
	\Edge(V-4_2)(V-5_1)
	\Edge(V5_1)(V6_2)
	\Edge(V5_1)(V6_0)
	\Edge(V-4_0)(V-5_1)
	\Edge(V-4_0)(V-5_-1)
	\Edge(V5_-1)(V6_0)
	\Edge(V5_-1)(V6_-2)
	\Edge(V-4_-2)(V-5_-1)
	\Edge(V-4_-2)(V-5_-3)
	\Edge(V5_-3)(V6_-2)
	\Edge(V6_4)(V5_3)
	\Edge(V6_2)(V5_3)
	\Edge(V6_2)(V5_1)
	\Edge(V6_0)(V5_1)
	\Edge(V6_0)(V5_-1)
	\Edge(V6_-2)(V5_-1)
	\Edge(V6_-2)(V5_-3)

    \tikzset{VertexStyle/.append style={orange, diamond, minimum size=1pt}}
	
    \Vertex[x=-1.500000,y=1.000000,NoLabel=true]{V-2_1}

    \tikzset{EdgeStyle/.append style={line width=5pt}}   

	\Edge(V-2_1)(V-2_2)
	\Edge(V-2_1)(V-2_0)
	\Edge(V-2_1)(V0_1)
	
    \Edge(V-2_2)(V-2_3)
    \Edge(V-2_0)(V-2_-1)
    \Edge(V0_1)(V1_2)

    \Edge(V-2_3)(V-2_4)
    \Edge(V-2_-1)(V-2_-2)
    \Edge(V1_2)(V2_2)

    \Edge(V-2_-2)(V-2_-3)
    \Edge(V2_2)(V4_1)
    
    \Edge(V4_1)(V5_1)
    
    \Edge(V5_1)(V6_2)
    
    \tikzset{VertexStyle/.append style={orange, diamond, minimum size=12pt}}
	
    \Vertex[x=-1.500000,y=1.000000,NoLabel=true]{V-2_1}

\end{tikzpicture}
\caption{The three paths used to show that two firefighters do not suffice on the infinite hexagonal grid when $d=1$.}
    \label{fig:hex_paths}
\end{figure}    
   
    Now observe that in the case of $d=2$ the fire is contained in five turns as illustrated in Figure~\ref{fig:hex_d2_2f}. For any distance greater than two the same strategy can be applied and thus two firefighters are also sufficient when $d > 2$.

\begin{lemma} \label{lemma:hex2ffd2}
    Two firefighters suffice to contain the fire on the infinite hexagonal grid when $d \geq 2$.
\end{lemma}

\begin{proof}
    See Figure~\ref{fig:hex_d2_2f}.
\end{proof}

\begin{figure}[H]
    \centering
\begin{tikzpicture}[scale=0.5]
    \SetUpVertex[FillColor=white]

    \tikzset{VertexStyle/.append style={minimum size=8pt, inner sep=1pt}}

	\Vertex[x=0.000000,y=3.000000,NoLabel=true]{V0_3}
	\Vertex[x=0.000000,y=1.000000,NoLabel=true]{V0_1}
	\Vertex[x=0.000000,y=-1.000000,NoLabel=true]{V0_-1}
	\Vertex[x=0.000000,y=-3.000000,NoLabel=true]{V0_-3}
	\Vertex[x=1.000000,y=4.000000,NoLabel=true]{V1_4}
	\Vertex[x=1.000000,y=2.000000,NoLabel=true]{V1_2}
	\Vertex[x=1.000000,y=0.000000,NoLabel=true]{V1_0}
	\Vertex[x=1.000000,y=-2.000000,NoLabel=true]{V1_-2}
	\Vertex[x=-1.500000,y=3.000000,NoLabel=true]{V-2_3}
	\Vertex[x=2.500000,y=4.000000,NoLabel=true]{V2_4}
	\Vertex[x=-1.500000,y=1.000000,NoLabel=true]{V-2_1}
	\Vertex[x=2.500000,y=2.000000,NoLabel=true]{V2_2}
	\Vertex[x=-1.500000,y=-1.000000,NoLabel=true]{V-2_-1}
	\Vertex[x=2.500000,y=0.000000,NoLabel=true]{V2_0}
	\Vertex[x=-1.500000,y=-3.000000,NoLabel=true]{V-2_-3}
	\Vertex[x=2.500000,y=-2.000000,NoLabel=true]{V2_-2}
	\Vertex[x=-2.500000,y=4.000000,NoLabel=true]{V-2_4}
	\Vertex[x=3.500000,y=3.000000,NoLabel=true]{V4_3}
	\Vertex[x=-2.500000,y=2.000000,NoLabel=true]{V-2_2}
	\Vertex[x=3.500000,y=1.000000,NoLabel=true]{V4_1}
	\Vertex[x=-2.500000,y=0.000000,NoLabel=true]{V-2_0}
	\Vertex[x=3.500000,y=-1.000000,NoLabel=true]{V4_-1}
	\Vertex[x=-2.500000,y=-2.000000,NoLabel=true]{V-2_-2}
	\Vertex[x=3.500000,y=-3.000000,NoLabel=true]{V4_-3}
	\Vertex[x=-4.000000,y=4.000000,NoLabel=true]{V-4_4}
	\Vertex[x=5.000000,y=3.000000,NoLabel=true]{V5_3}
	\Vertex[x=-4.000000,y=2.000000,NoLabel=true]{V-4_2}
	\Vertex[x=5.000000,y=1.000000,NoLabel=true]{V5_1}
	\Vertex[x=-4.000000,y=0.000000,NoLabel=true]{V-4_0}
	\Vertex[x=5.000000,y=-1.000000,NoLabel=true]{V5_-1}
	\Vertex[x=-4.000000,y=-2.000000,NoLabel=true]{V-4_-2}
	\Vertex[x=5.000000,y=-3.000000,NoLabel=true]{V5_-3}
	\Vertex[x=-5.000000,y=3.000000,NoLabel=true]{V-5_3}
	\Vertex[x=6.000000,y=4.000000,NoLabel=true]{V6_4}
	\Vertex[x=-5.000000,y=1.000000,NoLabel=true]{V-5_1}
	\Vertex[x=6.000000,y=2.000000,NoLabel=true]{V6_2}
	\Vertex[x=-5.000000,y=-1.000000,NoLabel=true]{V-5_-1}
	\Vertex[x=6.000000,y=0.000000,NoLabel=true]{V6_0}
	\Vertex[x=-5.000000,y=-3.000000,NoLabel=true]{V-5_-3}
	\Vertex[x=6.000000,y=-2.000000,NoLabel=true]{V6_-2}
	\Edge(V0_3)(V-2_3)
	\Edge(V0_3)(V1_4)
	\Edge(V0_3)(V1_2)
	\Edge(V0_1)(V-2_1)
	\Edge(V0_1)(V1_2)
	\Edge(V0_1)(V1_0)
	\Edge(V0_-1)(V-2_-1)
	\Edge(V0_-1)(V1_0)
	\Edge(V0_-1)(V1_-2)
	\Edge(V0_-3)(V-2_-3)
	\Edge(V0_-3)(V1_-2)
	\Edge(V1_4)(V2_4)
	\Edge(V1_2)(V2_2)
	\Edge(V1_0)(V2_0)
	\Edge(V1_-2)(V2_-2)
	\Edge(V-2_3)(V-2_4)
	\Edge(V-2_3)(V-2_2)
	\Edge(V2_4)(V4_3)
	\Edge(V-2_1)(V-2_2)
	\Edge(V-2_1)(V-2_0)
	\Edge(V2_2)(V4_3)
	\Edge(V2_2)(V4_1)
	\Edge(V-2_-1)(V-2_0)
	\Edge(V-2_-1)(V-2_-2)
	\Edge(V2_0)(V4_1)
	\Edge(V2_0)(V4_-1)
	\Edge(V-2_-3)(V-2_-2)
	\Edge(V2_-2)(V4_-1)
	\Edge(V2_-2)(V4_-3)
	\Edge(V-2_4)(V-4_4)
	\Edge(V4_3)(V5_3)
	\Edge(V-2_2)(V-4_2)
	\Edge(V4_1)(V5_1)
	\Edge(V-2_0)(V-4_0)
	\Edge(V4_-1)(V5_-1)
	\Edge(V-2_-2)(V-4_-2)
	\Edge(V4_-3)(V5_-3)
	\Edge(V-4_4)(V-5_3)
	\Edge(V5_3)(V6_4)
	\Edge(V5_3)(V6_2)
	\Edge(V-4_2)(V-5_3)
	\Edge(V-4_2)(V-5_1)
	\Edge(V5_1)(V6_2)
	\Edge(V5_1)(V6_0)
	\Edge(V-4_0)(V-5_1)
	\Edge(V-4_0)(V-5_-1)
	\Edge(V5_-1)(V6_0)
	\Edge(V5_-1)(V6_-2)
	\Edge(V-4_-2)(V-5_-1)
	\Edge(V-4_-2)(V-5_-3)
	\Edge(V5_-3)(V6_-2)
	\Edge(V6_4)(V5_3)
	\Edge(V6_2)(V5_3)
	\Edge(V6_2)(V5_1)
	\Edge(V6_0)(V5_1)
	\Edge(V6_0)(V5_-1)
	\Edge(V6_-2)(V5_-1)
	\Edge(V6_-2)(V5_-3)

    \tikzset{VertexStyle/.append style={orange, diamond, minimum size=12pt, text=black, inner sep=0.5pt}}
	
    \Vertex[x=-1.500000,y=1.000000,L={\scriptsize 0}]{V0_1}

    \tikzset{VertexStyle/.append style={circle, minimum size=8pt, inner sep=1pt}}

    \Vertex[x=0.00000,y=1.000000,L={\scriptsize 1}]{V0_1}

    \Vertex[x=1.00000,y=2.000000,L={\scriptsize 2}]{V0_1}
    \Vertex[x=1.00000,y=0.000000,L={\scriptsize 2}]{V0_1}

    \Vertex[x=2.50000,y=2.000000,L={\scriptsize 3}]{V0_1}
    \Vertex[x=2.50000,y=0.000000,L={\scriptsize 3}]{V0_1}

    \Vertex[x=3.50000,y=3.000000,L={\scriptsize 4}]{V0_1}
    \Vertex[x=3.50000,y=1.000000,L={\scriptsize 4}]{V0_1}
    \Vertex[x=3.50000,y=-1.000000,L={\scriptsize 4}]{V0_1}

    \tikzset{VertexStyle/.append style={black, star, text=white, minimum size=12pt, inner sep=0.5pt}}
   
    \Vertex[x=-2.500000,y=2.000000,L={\scriptsize 0}]{V0_1}
    \Vertex[x=-2.500000,y=0.000000,L={\scriptsize 0}]{V0_1}

    \tikzset{VertexStyle/.append style={rectangle, minimum size=8pt, inner sep=1pt}}

    \Vertex[x=0.000000,y=3.000000,L={\scriptsize 1}]{V0_1}
    \Vertex[x=0.000000,y=-1.000000,L={\scriptsize 1}]{V0_1}

    \Vertex[x=2.500000,y=4.000000,L={\scriptsize 2}]{V0_1}
    \Vertex[x=2.500000,y=-2.000000,L={\scriptsize 2}]{V0_1}

    \Vertex[x=5.000000,y=3.000000,L={\scriptsize 3}]{V0_1}
    \Vertex[x=5.000000,y=-1.000000,L={\scriptsize 3}]{V0_1}

    \Vertex[x=5.000000,y=1.000000,L={\scriptsize 4}]{V0_1}

\end{tikzpicture}
\caption{A strategy for containment with two firefighters on the infinite hexagonal grid when $d=2$.}
    \label{fig:hex_d2_2f}
\end{figure}
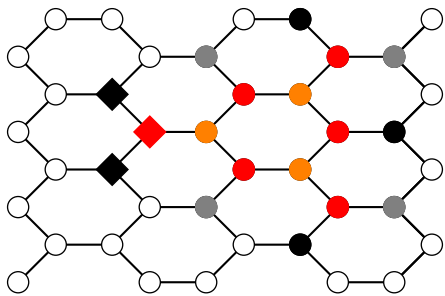    

    We also conjecture that two firefighters are necessary to contain the fire in the case of $d=2$ which we formalize in Conjecture~\ref{conj:two_nec}. This conjecture is a weaker version of Messinger's Conjecture, since if one firefighter is insufficient for the original game then it is certainly the case that one firefighter is insufficient for the distance-restricted game. Thus at least two firefighters would be necessary as stated in Conjecture~\ref{conj:two_nec}.

\begin{restatable}[]{conjecture}{conjd}
\label{conj:two_nec}
    One firefighter does not suffice to contain the fire in the infinite hexagonal grid when $d=2$.
\end{restatable}

    This special case of Messinger's Conjecture has a nice property which should make a proof more attainable. Namely, whenever the firefighter moves a distance of two its start and end point have exactly one neighbour in common. This common neighbour vertex now has all but one of its neighbours protected, and so this vertex is essentially protected since the fire reaching this vertex does not change anything in terms of how much the fire will spread as it cannot spread from this vertex to any new vertices. So in essence, any strategy that the firefighter employs is equivalent to protecting the vertices of some walk in the grid. It is our belief that this property could be used to show that any strategy being successful implies that a spiraling strategy must also be successful, thus allowing a proof of this conjecture to be built from a proof that the spiraling strategy does not work.

\section{Discussion and Open Problems} \label{sec:conclusions}

    Throughout this paper we have seen many upper bounds on $f_d(G,u)$ for fixed $G$ and $d$, and progress has been made towards establishing some lower bounds to complement them. Lower bounds are of particular interest in general, and especially for this game. The proofs of these lower bounds cannot make use of widely used tools like Fogarty's Hall-type condition~\cite{fogarty2003catching} as the firefighters' positions are no longer a sequence of arbitrary moves. The way that every move affects all moves that follow can make the problem of distance-restricted firefighting more difficult than the original game.

    Alongside these questions about lower bounds for particular graphs, there are also questions about the game in general. For example we initially wondered if every infinite, planar\footnote{By planar here we mean every finite subgraph is planar. If we also require that the embedding in $\mathbb{R}^2$ has no accumulation points then the problem is still open.}, $k$-regular, vertex-transitive graph would require $k$ firefighters to contain the fire when $d=1$. However if we consider $P^{\mathbb{Z}}\:\square\:C^{n}$ for any $n \in \{3,4,5,\ldots\}$, observe that two firefighters can contain the fire by starting far enough away from the fire along $P^{\mathbb{Z}}$ in either direction and then forming a protective barrier in the form of two $n$-cycles. Moreover, for any vertex in the graph, the subgraph induced by the set of vertices within distance $\left\lfloor \frac{n-2}{2} \right\rfloor$ is isomorphic to the subgraph induced by the same process for any vertex in the infinite square grid.

    Another question we have is under which conditions $f_d(H,u) \leq f_d(G,u)$ for $H$ a subgraph of $G$ and $u \in V(H)$. We have seen that this does not hold in general, but if we restrict ourselves to the strong, square, hexagonal and our subdivided infinite hexagonal grid as $G$ and $H$ then the property does hold. All these examples of the inequality not holding have to do with the fact that, unlike in the original model, the set of vertices the firefighters can eventually reach is reduced if the set of burnt vertices is a vertex cut. Due to this fact, the inequality likely involves the connectedness of both $G$ and $H$, potentially in relation to the degree of their vertices. For example, all of our grids except the subdivided infinite hexagonal grid are $k$-regular and $k$-connected, but the subdivided infinite hexagonal grid is not regular. So it is possible that $G$ and $H$ being $k$-regular and $k$-connected (for possibly different values of $k$) could imply that $f_d(H,u) \leq f_d(G,u)$, but a different condition than regularity would be needed for a full characterization.

    The final more general question we pose is how the behaviour of the game changes if we look at the game with an infinite value of $d$ while still maintaining that the firefighters not move through the fire. Certainly the results from Theorem~\ref{thm:arb_diff_2} and Theorem~\ref{thm:arb_diff} still hold, but the behaviour of the game in general could sit somewhere in between the original game and the distance-restricted game with finite $d$. 

    We of course also have the conjectures we made through the paper which we reiterate here.

    \conja*
    \conjb*
    \conjc*
    \conjd*

    We consider Conjectures~\ref{conj:two_insuff} and~\ref{conj:three_quart} to be the more interesting conjectures as Conjecture~\ref{conj:column} is something that would likely be proven in order to then prove the other two conjectures from Section~\ref{sec:square}. Even if Conjecture~\ref{conj:three_quart} is unsolved, Conjecture~\ref{conj:two_insuff} would still be an interesting result as it would represent a new proof of a lower bound in firefighting that is non-trivial and does not use Fogarty's Hall-type condition. 

    Conjecture~\ref{conj:two_nec} would also be an interesting result, especially if Messinger's conjecture was shown to be false. If Messinger's conjecture were false, Conjecture~\ref{conj:two_nec} would remain open and would open up a new question of which values of $d$ permit a strategy where one firefighter can contain the fire on the infinite hexagonal grid. Alternatively, proving Conjecture~\ref{conj:two_nec} would imply that any counterexample to Messinger's conjecture would require the firefighter to move a distance greater than two at least once. In any case, Conjecture~\ref{conj:two_nec} could lead to some interesting explorations into how the game behaves when the firefighters' strategy can be considered as a walk.

\section*{Acknowledgements}

\noindent Authors A.C. Burgess and D.A. Pike acknowledge research support from NSERC Discovery Grants 2019-04328 and RGPIN-2022-03829, respectively. Author J. Marcoux acknowledges support from Memorial University of Newfoundland as a graduate student at the time this paper was initially written.

\section*{Additional Statements}

    \noindent Competing interests: The authors declare none.
    \par
    \noindent Data availability: There is no data associated with this article.

\appendixpage

\appendix
\section{Full Strategy for proof of Lemma~\ref{lemma:strong4ffd2}} \label{app:strong}

Below is the list of sets describing which vertices the firefighters need to defend to contain the fire on the infinite strong grid in the distance-restricted game with $d=2$. We assume the fire begins at $(13,8)$.

\begin{align*}
    &\textit{POS}_0 = \{(1,(12,9)),(1,(13,9)),(1,(14,9)),(1,(14,8))\}\\
    &\textit{POS}_1 = \{(1,(11,9)),(1,(15,8)),(1,(15,7)),(1,(15,6))\}\\
    &\textit{POS}_2 = \{(1,(10,9)),(1,(15,6)),(1,(15,5)),(1,(14,5))\}\\
    &\textit{POS}_3 = \{(1,(9,9)),(1,(14,4)),(1,(13,4)),(1,(12,4))\}\\
    &\textit{POS}_4 = \{(1,(8,9)),(1,(12,3)),(1,(11,3)),(1,(10,3))\}\\
    &\textit{POS}_5 = \{(1,(7,9)),(1,(10,2)),(1,(9,2)),(1,(8,2))\}\\
    &\textit{POS}_6 = \{(1,(6,9)),(1,(8,1)),(1,(7,1)),(1,(6,1))\}\\
    &\textit{POS}_7 = \{(1,(5,9)),(1,(6,1)),(1,(5,1)),(1,(5,2))\}\\
    &\textit{POS}_8 = \{(1,(4,9)),(1,(4,2)),(1,(4,3)),(1,(4,4))\}\\
    &\textit{POS}_9 = \{(1,(3,9)),(1,(3,4)),(1,(3,5)),(1,(3,6))\}\\
    &\textit{POS}_{10} = \{(1,(2,9)),(1,(2,6)),(1,(2,7)),(1,(2,8))\}
\end{align*}

\bibliographystyle{amsplainnodash}
\bibliography{bibliography}

\end{document}